\documentclass[10pt,reqno]{amsart}
\usepackage{graphicx}
\usepackage{hyperref}
\usepackage{verbatim}
\usepackage[usenames,dvipsnames]{color}

\newtheorem{thm}{Theorem}[section]
\newtheorem{lem}[thm]{Lemma}
\newtheorem{prop}[thm]{Proposition}
\newtheorem{cor}[thm]{Corollary}
\newtheorem{defn}[thm]{Definition}

\newtheorem{rem}[thm]{Remark}

\numberwithin{equation}{section}

\def\R{\mathbb{R}}

\def\Z{\mathbb{Z}}

\def\ra{\rightarrow}

\def\al{\alpha}

\def\ep{\epsilon}

\def\la{\lambda}

\def\si{\sigma}

\def\om{\omega}
\def\Om{\Omega}
\def\de{\delta}
\def\De{\Delta}
\def\ga{\gamma}

\begin{document}

\title[Existence, Uniqueness and Stability of Transition Fronts]{Existence, Uniqueness and Stability of Transition Fronts of Nonlocal Equations in Time Heterogeneous Bistable Media}

\author{Wenxian Shen}
\address{Department of Mathematics and Statistics, Auburn University, Auburn, AL 36849, USA}
\email{wenxish@auburn.edu}

\author{Zhongwei Shen}
\address{Department of Mathematical and Statistical Sciences,
University of Alberta, Edmonton, AB T6G 2G1, Canada}
\email{zhongwei@ualberta.ca}





\begin{abstract}
The present paper is devoted to the study of the existence, the uniqueness and the stability of transition fronts of nonlocal dispersal equations in time heterogeneous media of bistable type under the unbalanced condition. We first study space non-increasing transition fronts and prove various important qualitative properties, including  uniform steepness,  stability, uniform stability and exponential decaying estimates. Then, we show that any transition front, after certain space shift, coincides with a space non-increasing transition front (if it exists), which implies the uniqueness, up to space shifts, and monotonicity of transition fronts provided that a space non-increasing transition front exists. Moreover, we show that a transition front must be a periodic traveling wave in periodic media and asymptotic speeds of transition fronts exist in uniquely ergodic media. Finally, we prove the existence of space non-increasing transition fronts, whose proof does not need the unbalanced condition.
\end{abstract}

\maketitle

\tableofcontents


\section{Introduction}

This present paper is devoted to the investigation of transition fronts of  the following nonlocal dispersal equation in time heterogeneous media
\begin{equation}\label{main-eqn}
u_{t}=J\ast u-u+f(t,u),\quad (t,x)\in\R\times\R,
\end{equation}
where $J$ is a symmetric dispersal kernel function, $[J\ast u](t,x)=\int_{\R}J(x-y)u(t,y)dy$, and $f$ is a bistable type nonlinearity. More precisely, we
assume that $J$ and $f$  satisfy (H1)-(H3) stated in the following.

\begin{itemize}
\item[\bf (H1)] $J:\R\to\R$ is continuously differentiable and satisfies $J\not\equiv0$, $J(x)=J(-x)\geq0$ for $x\in\R$, $\int_{\R}J(x)dx=1$ and
\begin{equation*}
\int_{\R}J(x)e^{\ga x}dx<\infty,\quad\int_{\R}|J'(x)|e^{\ga x}dx<\infty,\quad\forall\ga\in\R.
\end{equation*}
\end{itemize}

\begin{itemize}
\item[\bf (H2)] There exist $C^{2}$ functions $f_{B}:\R\to\R$ and  $f_{\tilde{B}}:\R\to\R$ such that
\begin{equation*}
f_{B}(u)\leq f(t,u)\leq f_{\tilde{B}}(u),\quad (t,u)\in\R\times[0,1].
\end{equation*}
Moreover, the following conditions hold:
\begin{itemize}
\item $f:\R\times\R\to\R$ is continuously differentiable, and satisfies
\begin{equation*}
\sup_{(t,u)\in\R\times[-1,2]}\big(|f_{t}(t,u)|+|f_{u}(t,u)|\big)<\infty;
\end{equation*}

\item $f_{B}$ is of standard bistable type, that is, $f_{B}(0)=f_{B}(\theta)=f_{B}(1)=0$ for some $\theta\in(0,1)$, $f_{B}(u)<0$ for $u\in(0,\theta)$, $f_{B}(u)>0$ for $u\in(\theta,1)$ and satisfies the unbalanced condition
\begin{equation}\label{unbalanced-cond}
\text{the speed of traveling waves of $u_{t}=J\ast u-u+f_{B}(u)$ is positive};
\end{equation}

\item  $f_{\tilde{B}}$ is also of standard bistable type, that is, $f_{\tilde{B}}(0)=f_{\tilde{B}}(\tilde{\theta})=f_{\tilde{B}}(1)=0$ for some $\tilde{\theta}\in(0,1)$, $f_{\tilde{B}}(u)<0$ for $u\in(0,\tilde{\theta})$ and $f_{\tilde{B}}(u)>0$ for $u\in(\tilde{\theta},1)$.
\end{itemize}
\end{itemize}

We remark that (H2) implies that $f(t,0)=0=f(t,1)$ for all $t\in\R$, that is, $u\equiv0$ and $u\equiv1$ are two constant solutions of \eqref{main-eqn}, and the speed of traveling waves of
\begin{equation}\label{main-eqn-bi-f-B}
u_{t}=J\ast u-u+f_{B}(u)
\end{equation}
is unique, and traveling waves of \eqref{main-eqn-bi-f-B} are unique up to shifts (see \cite{BaFiReWa97}). Here, by traveling waves of \eqref{main-eqn-bi-f-B}, we mean global-in-time solutions of the form $\phi_{B}(x-c_{B}t)$ with $\phi_{B}(-\infty)=1$ and $\phi_{B}(\infty)=0$. Moreover, the unbalanced condition \eqref{unbalanced-cond} is equivalent to the speed of traveling waves of \eqref{main-eqn-bi-f-B} being nonzero and $\int_{0}^{1}f_{B}(u)du>0$. Note this is different from that in the classical case, where the speed of traveling waves of $u_{t}=u_{xx}+f_{B}(u)$ has the same sign as that of the integral $\int_{0}^{1}f_{B}(u)du$ (see e.g. \cite{ArWe75}).

The next assumption makes sure the uniform stability of $u\equiv0$ and $u\equiv1$.

\begin{itemize}
\item[\bf (H3)]There exist $\theta_{0}$, $\theta_{1}$ with $0<\theta_{0}<\tilde{\theta}\leq\theta<\theta_{1}<1$ and $\beta_{0}>0$, $\beta_{1}>0$ such that
\begin{equation*}
f_{u}(t,u)\leq-\beta_{0},\quad u\in[-1,\theta_{0}]\quad\text{and}\quad
f_{u}(t,u)\leq-\beta_{1},\quad u\in[\theta_{1},2]
\end{equation*}
for all $t\in\R$.
\end{itemize}

Sometime, we also assume that $f$ satisfies

\begin{itemize}
\item[\bf (H4)] The ODE
\begin{equation}\label{ode-instable}
u_t=f(t,u)
\end{equation}
has an entire solution $u_0:\R\to\R$ satisfying
\begin{itemize}
\item $0<\inf_{t\in\R}u_0(t)\leq\sup_{t\in\R}u_0(t)<1$;

\item there exists $0<\de_{0}\ll1$ such that
\begin{equation}\label{nondegeneracy-intro}
\inf_{t\in\R}\inf_{u\in[u_{0}(t)-\de_{0},u_{0}(t)+\de_{0}]}f_{u}(t,u)>0,
\end{equation}
\item for any $t_0\in\R$, $u_1\in (0,u_0(t_0))$ and $u_2\in (u_0(t_0),1)$, there holds
$$
u(t;t_0,u_1)\to 0,\quad u(t;t_0,u_2)\to 1\quad\text{as}\quad t-t_0\to \infty,
$$
where $u(t;t_0,u_i)$ are the solution of \eqref{ode-instable} with $u(t_0;t_0,u_i)=u_i$ ($i=1,2$).
\end{itemize}
\end{itemize}

Among others, equation \eqref{main-eqn} is used to model the evolution of population density of species with Allee effect. A typical example is $f(t,u)=u(u-\theta(t))(1-u)$ for appropriate $\theta(t)$. Solutions of particular interest are transition fronts connecting $0$ and $1$ due to their importance in describing extinction and persistence of population. In the case that $f(t,u)\equiv f(u)$ is independent of $t$, transition fronts are strongly related to traveling waves, that is, solutions of the form $u(t,x)=\phi(x-ct)$ for some $\phi:\R\to (0,1)$ and $c\in\R$ with $\phi(-\infty)=1$ and $\phi(\infty)=0$. The reader is referred to \cite{BaFiReWa97, Ch97} for the study of the existence, the uniqueness and the stability of traveling waves of \eqref{main-eqn} in time independent bistable media. Also, see \cite{AcKu15,BaCh02,BaCh06,Ya09} and references therein for more related works. In \cite{FCh}, time almost-periodic traveling waves of \eqref{main-eqn} in the present of random diffusion are studied when $f(t,u)$ is almost periodic in $u$. As far as general time heterogeneity is concerned, there is little study on transition fronts of \eqref{main-eqn} with bistable nonlinearity.

The objective of this paper is to study the existence, the uniqueness and the stability of transition fronts of \eqref{main-eqn} when $f$ is a general time dependent function satisfying (H2) and (H3). We recall the definition of transition fronts.

\begin{defn}\label{defn-tf}
Suppose $f(t,0)=0=f(t,1)$ for all $t\in\R$. A global-in-time solution $u(t,x)$ of \eqref{main-eqn} is called a (right-moving) {\rm transition front} (connecting $0$ and $1$) in the sense of Berestycki-Hamel (see \cite{BeHa07,BeHa12}, also see \cite{Sh99-1,Sh99-2}) if $u(t,x)\in(0,1)$ for all $(t,x)\in\R\times\R$ and there exists a function $X:\R\to\R$, called {\rm interface location function}, such that
\begin{equation*}
\lim_{x\to-\infty}u(t,x+X(t))=1\,\,\text{and}\,\,\lim_{x\to\infty}u(t,x+X(t))=0\,\,\text{uniformly in}\,\,t\in\R.
\end{equation*}
\end{defn}

The notion of a transition front is a proper generalization of a traveling wave in homogeneous media or a periodic (or pulsating) traveling wave in periodic media. The interface location function $X(t)$ tells the position of the transition front $u(t,x)$ as time $t$ elapses. Notice, if $\xi(t)$ is a bounded function, then $X(t)+\xi(t)$ is also an interface location function. Thus, interface location function is not unique. But, it is easy to check that if $Y(t)$ is another interface location function, then $X(t)-Y(t)$ is a bounded function. Hence, interface location functions are unique up to addition by bounded functions. The uniform-in-$t$ limits (the essential property in the definition) shows the \textit{bounded interface width}, that is,
\begin{equation}\label{bounded-interface-width-defn}
\forall\,\,0<\ep_{1}\leq\ep_{2}<1,\quad\sup_{t\in\R}{\rm diam}\{x\in\R|\ep_{1}\leq u(t,x)\leq\ep_{2}\}<\infty.
\end{equation}
This actually gives an equivalent definition of transition fronts, that is, a global-in-time solution $u(t,x)$ of \eqref{main-eqn} is called a transition front if $u(t,x)\in(0,1)$ for all $(t,x)\in\R\times\R$, $u(t,x)\to1$ as $x\to-\infty$ and $u(t,x)\to0$ as $x\to\infty$ for all $t\in\R$, and \eqref{bounded-interface-width-defn} holds.

In the study of the existence, the stability and the uniqueness of transition fronts of \eqref{main-eqn}, sub- and super-solutions and comparison principles play crucial roles. We remark that showing a function constructed from a transition front is a sub-solution or a super-solution usually involves the space derivative of the transition front. However,  neither the definition nor the equation \eqref{main-eqn} guarantees any space regularity of transition fronts. In \cite{ShSh15-regularity}, we studied the space regularity of transition fronts of nonlocal dispersal equations in general heterogeneous media. The following proposition follows from \cite[Theorems 1.1 and Corollary 1.6]{ShSh15-regularity}.

\begin{prop}\label{prop-regularity-of-tf}
Assume (H1)-(H3). Let $u(t,x)$ be an arbitrary transition front of \eqref{main-eqn} and $X(t)$ be its interface location function. Then,
\begin{itemize}
\item[\rm(i)] there exists a continuously differentiable function $\tilde{X}:\R\to\R$ satisfying
$$c_{\min}\leq\dot{\tilde{X}}(t)\leq c_{\max},\quad \forall t\in\R$$
for some $0<c_{\min}\leq c_{\max}<\infty$ such that
$$\sup_{t\in\R}|X(t)-\tilde{X}(t)|<\infty;$$
in particular, $\tilde{X}(t)$ is also an interface location function of $u(t,x)$;

\item[\rm(ii)] $u(t,x)$ is regular in space, that is, $u(t,x)$ is continuously differentiable in $x$ for any $t\in\R$ and satisfies $$\sup_{(t,x)\in\R\times\R}|u_{x}(t,x)|<\infty.$$
\end{itemize}
\end{prop}

We point out that Proposition \ref{prop-regularity-of-tf} highly relies on the unbalanced condition \eqref{unbalanced-cond}. Replacing \eqref{unbalanced-cond} by the speed of traveling waves of \eqref{main-eqn-bi-f-B} being nonnegative, Proposition \ref{prop-regularity-of-tf} fails when \eqref{main-eqn-bi-f-B} admits discontinuous traveling waves with zero speed (see \cite{BaFiReWa97} for the sufficient and necessary condition). Whether Proposition \ref{prop-regularity-of-tf} holds when \eqref{main-eqn-bi-f-B} admits continuous traveling waves with zero speed leaves an interesting open question.

By Proposition \ref{prop-regularity-of-tf}, without loss of generality, we may then assume that the interface location function $X(t)$ of a transition front
$u(t,x)$ is continuously differentiable and satisfies
\begin{equation}
\label{interface-eq}
c_{\min}\leq\dot{X}(t)\leq c_{\max},\quad \forall t\in\R
\end{equation}
for some $0<c_{\min}\leq c_{\max}<\infty$. This shows the rightward propagation nature of transition fronts in the sense of Definition \ref{defn-tf}.

By general semigroup theory (see e.g. \cite{Paz}) and comparison principles,  for any $u_0\in C_{\rm unif}^b(\R,\R)$ and $t_{0}\in\R$, \eqref{main-eqn} has a unique global solution $u(t,\cdot;t_{0},u_0)\in C_{\rm unif}^b(\R,\R)$ with $u(t_{0},\cdot;t_{0},u_0)=u_0$, where
$$
C_{\rm unif}^b(\R,\R)=\left\{u\in C(\R,\R)\,:\, u\,\, \text{is uniformly continuous on}\,\,\R\,\,\text{and}\,\, \sup_{x\in\R}|u(x)|<\infty\right\}
$$
equipped with the norm $\|u\|:=\sup_{x\in\R}|u(x)|$.

Throughout this paper, we assume (H1)-(H3).
Among others, we prove in this paper the following results.

\begin{itemize}

\item[\rm(i)] (Uniform steepness)  Assume that $u(t,x)$ is
 a space non-increasing transition front of \eqref{main-eqn} with $X(t)$ being its interface location function.
For any $M>0$, there holds $$\sup_{t\in\R}\sup_{|x-X(t)|\leq M}u_{x}(t,x)<0$$
 (see Theorem \ref{thm-uniform-steepness}).

\item[\rm(ii)] (Uniform exponential stability) Assume that $u(t,x)$ is
a  space non-increasing transition front of \eqref{main-eqn}. Let $\{u_{t_{0}}\}_{t_{0}\in\R}$ be a family of uniformly continuous initial data satisfying
\begin{equation*}
u(t_{0},x-\xi_{0}^{-})-\mu_{0}\leq u_{t_{0}}(x)\leq u(t_{0},x-\xi_{0}^{+})+\mu_{0},\quad x\in\R,\,\,t_{0}\in\R
\end{equation*}
for $\xi_{0}^{\pm}\in\R$ and $\mu_{0}\in(0,\min\{\theta_{0},1-\theta_{1}\})$ being independent of $t_{0}\in\R$. Then, there exist $t_{0}$-independent constants $C>0$ and $\om_{*}>0$, and a family of shifts $\{\xi_{t_{0}}\}_{t_{0}\in\R}\subset\R$ satisfying $\sup_{t_{0}\in\R}|\xi_{t_{0}}|<\infty$ such that
\begin{equation*}
\sup_{x\in\R}|u(t,x;t_{0},u_{t_{0}})-u(t,x-\xi_{t_{0}})|\leq Ce^{-\om_{*}(t-t_{0})}
\end{equation*}
for all $t\geq t_{0}$ and $t_{0}\in\R$ (see Theorem \ref{thm-stability}).

\item[\rm(iii)] (Exponential decaying estimates) Assume that $u(t,x)$ is
 a space non-increasing transition front of \eqref{main-eqn} with  $X(t)$ being its interface location function.
 There exist two exponents $c_{\pm}>0$ and two shifts $h_{\pm}>0$ such that
\begin{equation*}
u(t,x+X(t)+h_{+})\leq e^{-c_{+}x}\quad\text{and}\quad u(t,x+X(t)-h_{-})\geq 1-e^{c_{-}x}
\end{equation*}
for all $(t,x)\in\R\times\R$ (see Theorem \ref{thm-expo-decaying}).

\item[\rm(iv)] (Uniqueness and monotonicity)
If $u(t,x)$ and $v(t,x)$ are two transition fronts  of \eqref{main-eqn}
with $u(t,x)$ being non-increasing in $x$, then there exists a shift $\xi\in\R$ such that
$$v(t,x)=u(t,x+\xi),\quad \forall(t,x)\in\R\times\R$$
and hence $v(t,x)$ is also non-increasing in $x$ (see Theorem \ref{thm-uniqeness}).

\item[\rm(v)] (Periodicity) If $u(t,x)$ is a space non-increasing transition front of \eqref{main-eqn} and, in addition, $f(t,u)$ is periodic in $t$, then $u(t,x)$ is a periodic traveling wave (see Theorem \ref{thm-asymptotic-speeds}(i)).

\item[\rm(vi)] (Asymptotic speeds) If $u(t,x)$ is a space non-increasing transition front of \eqref{main-eqn} and, in addition, $f(t,u)$ is uniquely ergodic, then $\lim_{t\to\pm\infty}\frac{X(t)}{t}$ exist (see Theorem \ref{thm-asymptotic-speeds}(ii)).

\item[\rm (vii)] (Existence)  Assume, in addition, (H4). There is a space non-increasing transition front of \eqref{main-eqn} (see Theorem \ref{thm-transition-front}).
\end{itemize}

We make some remarks about the above results.
\begin{enumerate}
\item From (i)-(iv), we see that if \eqref{main-eqn} admit a space non-increasing transition front under assumptions (H1)-(H3), then transition fronts of \eqref{main-eqn} are non-increasing in space, exponentially stable, exponentially decaying and unique up to space shifts. We point out that it can be shown that any transition front of corresponding reaction-diffusion equations in time heterogeneous media is non-increasing in space (see e.g. \cite{Sh99-1,ShSh14-1}), while it is not an easy job for nonlocal dispersal equations partly due to the lack of Harnack's inequality.

\item Note that the nonlinearity $f(t,u)$ satisfying (H2) and (H3) are bistable only in the general sense. For each $t\in\R$, $f(t,\cdot)$ may not be of bistable type. In particular, multiple zeros between $0$ and $1$ are allowed. It is not known (even in the reaction-diffusion equation case) whether the assumptions (H2) and (H3) on $f(t,u)$ are sufficient for the existence of space non-increasing transition fronts, which is guaranteed under the additional assumption (H4). This is given in (vii).

\item The establishment of the uniform exponential stability in (ii) in this general form is the most important result of the present paper, and the applicability of the uniform exponential stability to arbitrary transition fronts and other families of initial data is the key to the proof of (iii) and (iv), and then to that of (v) and (vi). We remark that for reaction-diffusion equations, the usual exponential stability instead of the uniform exponential stability, together with standard arguments using parabolic regularity, comparison principles and Harnack's inequality, are sufficient for various qualitative properties such as exponential decaying estimate and uniqueness (see e.g. \cite{MNRR09,ShSh14-1}). But for nonlocal equations, the standard arguments do not work very well, since we are lack of enough regularity, the comparison principles are not as flexible as that for reaction-diffusion equations, and Harnack's inequality is not known in the nonlocal case.

\item The proof of (vii) actually does not need the unbalanced condition \eqref{unbalanced-cond}. This is because we take a perturbation approach, that is, we consider the perturbed equation
\begin{equation}\label{eqn-perturbed-intro}
u_{t}=J\ast u-u+\ep u_{xx}+f(t,u),\quad (t,x)\in\R\times\R,
\end{equation}
and take advantage of the fact that the existence of transition fronts of \eqref{eqn-perturbed-intro} does not need \eqref{unbalanced-cond}. Of course, without \eqref{unbalanced-cond}, constructed transition fronts of \eqref{main-eqn} may not be continuous in space as mentioned after Proposition \ref{prop-regularity-of-tf}. It would be interesting to study qualitative properties of transition fronts in the absence of \eqref{unbalanced-cond}.

\item It should be pointed out that (H2) can also be applied to a general bistable nonlinearity $f(t,u)$ with the speed of traveling waves of $u_{t}=J\ast u-u+f_{\tilde{B}}(u)$ being negative. In fact, let $v(t,x)=1-u(t,x)$. Then, $v(t,x)$ satisfies
$$
v_t=J\ast v-v+\tilde f(t,v),\quad (t,x)\in\R\times\R,
$$
where $\tilde f(t,v)=-f(t,1-v)$. Hence
$$
\tilde f_{B}(v)\le \tilde f(t,v)\le \tilde f_{\tilde B}(v),\quad (t,v)\in\R\times\R\times[0,1]
$$
where $\tilde f_{B}(v)=-f_{\tilde B}(1-v)$ and $\tilde f_{\tilde B}(v)=-f_{B}(1-v)$.
Clearly, $\tilde f_{B}(\cdot)$ and $\tilde f_{\tilde B}(\cdot)$ are two standard bistable nonlinearities and the speed of traveling waves of $u_{t}=J\ast u-u+\tilde{f}_{B}(u)$ is positive.

\end{enumerate}

We remark that transition fronts can be defined in the same way for more general equations, say,
\begin{equation}\label{eqn-general}
u_{t}=J\ast u-u+f(t,x,u).
\end{equation}
Equation \eqref{eqn-general} in  various homogeneous media, i.e., $f(t,x,u)=f(u)$ with various types of nonlinearity $f(\cdot)$, is well studied. We refer to \cite{BaFiReWa97,CaCh04,Ch97,CoDu05,CoDu07,Sc80} and references therein for results concerning traveling waves. There are also some results concerning periodic traveling waves in periodic media of monostable type (see e.g. \cite{CDM13,DiLi15,RaShZh,ShZh10,ShZh12-1,ShZh12-2}). The study of \eqref{eqn-general} in general heterogeneous media is very recent and results concerning front propagation are very limited. In \cite{BCV14}, Berestycki, Coville and Vo studied principal eigenvalue, positive solution and long-time behavior of solutions of \eqref{eqn-general} in the space heterogeneous monostable media. In \cite{LiZl14},  Lim and Zlato\v{s} also studied \eqref{eqn-general} in the space heterogeneous monostable media, but with different settings, and proved the existence of transition fronts. In \cite{BeRo}, Berestycki and Rodriguez studied propagation and blocking phenomenon of \eqref{eqn-general} with barrier nonlinearities in space heterogeneous media of bistable type.  In \cite{ShSh14-2,ShSh14-3}, the authors of the present paper studied \eqref{eqn-general} in time heterogeneous media of ignition type and proved the existence, regularity and stability of transition fronts.

We end the introduction by mentioning some relevant results about reaction-diffusion equations and discrete equations in bistable media, that is,
\begin{equation}\label{eqn-random}
u_{t}=u_{xx}+f(t,x,u),\quad (t,x)\in\R\times\R
\end{equation}
and
\begin{equation}\label{eqn-discrete}
\dot{u}_{i}=u_{i+1}-2u_{i}+u_{i-1}+f(t,i,u),\quad (t,i)\in\R\times\Z,
\end{equation}
where $f$ in both cases is of bistable type. As a classical model, \eqref{eqn-random} has been attracting extensive studies and results concerning front propagation are quite complete except in the most general case, i.e., $f(t,x,u)$ depends generally on both $t$ and $x$ (see \cite{AlBaCh99,ArWe75,ArWe78,BeHa02,DiHaZh14,DiHaZh15,DuMa10,FiMc77,Ha14,LeKe00,MuZh13,Na14,Po11,Po15,Sh99-1,Sh99-2,Sh04,Sh06,We02,Zl06,Zl15}) and references therein. As \eqref{eqn-general} in the bistable case, not a lot is known about \eqref{eqn-discrete}. We refer the readers to \cite{CMV99,CMS98,Ma99,Zi91,Zi92} and references therein for works in homogeneous media, and to \cite{ChGuWu08,Sh03} for works in periodic media.

The rest of the paper is organized as follows. In Section 2, we focus our study on uniform steepness of space non-increasing transition fronts of  \eqref{main-eqn}. We investigate uniform exponential  stability and exponential decaying estimates of space non-increasing transition fronts
of \eqref{main-eqn} in Sections 3 and 4, respectively. In Section \ref{sec-uniqueness}, we show that any transition front of equation \eqref{main-eqn}, after certain space shift, coincides with a space non-increasing transition front (if it exists). In Section \ref{sec-ptw}, under the additional time periodic assumption on the nonlinearity, we show that any transition front must be a periodic traveling wave. Under the assumption that the nonlinearity $f(t,u)$ is uniquely ergodic, we show that asymptotic speeds of transition fronts exist. In Section \ref{sec-construction-tf}, we prove the existence of space non-increasing transition fronts of \eqref{main-eqn}. In Appendix \ref{sec-app-cp}, we state some comparison principles.



\section{Uniform steepness of space non-increasing transition fronts}\label{sec-uniform-steepness}

In this section, we study the uniform steepness of space non-increasing transition fronts of
 \eqref{main-eqn}. Throughout this section, we assume (H1)-(H3).

Suppose that $u(t,x)$ is a transition front. For $\la\in(0,1)$, let $X_{\la}^{-}(t)$ and $X_{\la}^{+}(t)$ be the leftmost and rightmost interface locations at $\la$, that is,
\begin{equation}\label{defn-interface-locations}
X_{\la}^{-}(t)=\inf\{x\in\R|u(t,x)\leq\la\}\quad\text{and}\quad X_{\la}^{+}(t)=\sup\{x\in\R|u(t,x)\geq\la\}.
\end{equation}
Trivially, $X_{\la}^{-}(t)\leq X_{\la}^{+}(t)$ and $X^{\pm}(t)$ are non-increasing in $\la$. We see that it may happen that $u(t,X_{\la}^{-}(t))>\la$ or $u(t,X_{\la}^{+}(t))<\la$ due to possible jumps. But, it is clear that
$u(t,x)>\la$ for $x<X_{\la}^{-}(t)$ and $u(t,x)<\la$ for $x>X_{\la}^{+}(t)$.

In what follows in this section, $u(t,x)$ will be an arbitrary transition front of \eqref{main-eqn} that is non-increasing in space, i.e., $u_{x}(t,x)\leq0$ for $(t,x)\in\R\times\R$ (recall that by Proposition \ref{prop-regularity-of-tf} any transition front is continuously differentiable in space). By comparison principle, $u(t,x)$ is decreasing in $x$ for any $t\in\R$. As a result, for any $\la\in(0,1)$, the leftmost and rightmost interface locations coincide, i.e., $X_{\la}^{+}(t)=X_{\la}^{-}(t)$, which will be denoted by $X_{\la}(t)$. In particular, $u(t,X_{\la}(t))=\la$. Let $X(t)$ be the interface location function corresponding to $u(t,x)$. Without loss of generality, we assume that $X(t)$ satisfies \eqref{interface-eq}.

The main result in this section is given in

\begin{thm}\label{thm-uniform-steepness}
For any $M>0$, there holds
$$\sup_{t\in\R}\sup_{|x-X(t)|\leq M}u_{x}(t,x)<0.$$
\end{thm}

To prove Theorem \ref{thm-uniform-steepness}, we first prove two lemmas.
The first lemma  follows  directly from
 the definition of transition fronts.

\begin{lem}\label{lem-bounded-interface-width}
For any $\la\in(0,1)$, there hold
$$\sup_{t\in\R}|X(t)-X^{\pm}_{\la}(t)|<\infty.$$
\end{lem}
\begin{proof}
By the uniform-in-$t$ limits $\lim_{x\to-\infty}u(t,x+X(t))=1$ and $\lim_{x\to\infty}u(t,x+X(t))=0$, there exist $x_{1}$ and $x_{2}$ such that $u(t,x+X(t))>\la$ for all $x\leq x_{1}$ and $t\in\R$, and $u(t,x+X(t))<\la$ for all $x\geq x_{2}$ and $t\in\R$. It then follows from the definition of $X_{\la}^{\pm}(t)$ that $x_{1}+X(t)\leq X_{\la}^{-}(t)$ and $x_{2}+X(t)\geq X_{\la}^{+}(t)$ for all $t\in\R$. In particular,
\begin{equation*}
x_{1}+X(t)\leq X_{\la}^{-}(t)\leq X_{\la}^{+}(t)\leq x_{2}+X(t),\quad t\in\R.
\end{equation*}
This completes the proof.
\end{proof}

We remark that the monotonicity of $u(t,x)$ in $x$ is not required in the above lemma, which is true for an arbitrary transition front. That is why we used $X^{\pm}_{\la}(t)$ instead of $X_{\la}(t)$.

As a simple consequence of implicit function theorem, the equation $u(t,X_{\la}(t))=\la$, Lemma \ref{lem-bounded-interface-width} and Theorem \ref{thm-uniform-steepness}, we find

\begin{cor}\label{cor-diff-interface-mono}
For any $\la\in(0,1)$, $X_{\la}(t)$ is continuously differentiable and satisfies
\begin{equation*}
\dot{X}_{\la}(t)=-\frac{u_{t}(t,X_{\la}(t))}{u_{x}(t,X_{\la}(t))},\quad\forall t\in\R\quad\text{and}\quad\sup_{t\in\R}|\dot{X}_{\la}(t)|<\infty.
\end{equation*}
\end{cor}

Since $u_{t}(t,x)$ changes signs in general due to the time dependence of $f(t,u)$, $\dot{X}_{\la}(t)$ changes its signs. Thus, in general, transition fronts in the present case move to the right with oscillations.

The next lemma inspired by \cite[Theorem 5.1]{Ch97} and \cite[Lemma 3.2]{Sh99-1} is crucial to uniform steepness. We refer the reader to Appendix \ref{sec-app-cp} for comparison principles.

\begin{lem}\label{lem-tech-1234567}
Let $u_{1}(t,x;\tau)$ and $u_{2}(t,x;\tau)$ be sub-solution and super-solution of \eqref{main-eqn}, respectively, and satisfy
\begin{equation*}
-1\le u_{1}(t,x;\tau)\leq u_{2}(t,x;\tau)\le 2,\quad x\in\R,\quad t\geq\tau.
\end{equation*}
Then, for any $t>t_{0}\geq\tau$, $h>0$ and $z\in\R$, there holds
\begin{equation*}
u_{1}(t,x;\tau)-u_{2}(t,x;\tau)\leq C\int_{z-h}^{z+h}[u_{1}(t_{0},y;\tau)-u_{2}(t_{0},y;\tau)]dy,\quad x\in\R,
\end{equation*}
where $C=C(t-t_{0},|x-z|,h)>0$ satisfies
\begin{itemize}
\item[\rm(i)] $C\to0$ polynomially as $t-t_{0}\to0$ and $C\to0$ exponentially as $t-t_{0}\to\infty$;
\item[\rm(ii)] $C:(0,\infty)\times[0,\infty)\times(0,\infty)\to(0,\infty)$ is locally uniformly positive in the sense that for any $0<t_{1}<t_{2}<\infty$, $M_{1}>0$ and $h_{1}>0$, there holds
\begin{equation*}
\inf_{t\in[t_{1},t_{2}], M\in[0,M_{1}], h\in(0,h_{1}]} C(t,M,h)>0.
\end{equation*}
\end{itemize}
\end{lem}
\begin{proof}
Let $t>t_{0}\geq\tau$. Set $v_{1}(t,x):=u_{1}(t,x;\tau)$ and $v_{2}(t,x):=u_{2}(t,x;\tau)$. By assumption, $v(t,x):=v_{1}(t,x)-v_{2}(t,x)\leq0$ and satisfies
\begin{equation*}
v_{t}\leq J\ast v-v+f(t,v_{1})-f(t,v_{2}).
\end{equation*}
By $\rm(H2)$, we can find $K>0$ such that $f(t,v_{1})-f(t,v_{2})\leq -K(v_{1}-v_{2})$, which implies that
$$v_{t}\leq J\ast v-v-Kv.$$
Setting $\tilde{v}(t,x):=e^{(1+K)(t-t_{0})}v(t,x)\leq0$, we see
\begin{equation}\label{an-differential-inequality-1}
\tilde{v}_{t}\leq J\ast\tilde{v}\leq0.
\end{equation}
In particular, $\tilde{v}(t,x)\leq\tilde{v}(t_{0},x)$. It then follows that
$$\tilde{v}_{t}(t,x)\leq [J\ast\tilde{v}(t,\cdot)](x)\leq[J\ast\tilde{v}(t_{0},\cdot)](x).$$
Integrating over $[t_{0},t]$ with respect to the time variable, we find from $\tilde{v}(t_{0},x)\leq0$ that
\begin{equation*}
\tilde{v}(t,x)\leq (t-t_{0})[J\ast\tilde{v}(t_{0},\cdot)](x)+\tilde{v}(t_{0},x)\leq(t-t_{0})[J\ast\tilde{v}(t_{0},\cdot)](x).
\end{equation*}
In particular, for any $T>0$, we have
\begin{equation}\label{initial-estimate}
\tilde{v}(t_{0}+T,x)\leq T[J\ast\tilde{v}(t_{0},\cdot)](x).
\end{equation}

Then, considering \eqref{an-differential-inequality-1} with initial time at $t_{0}+T$ and repeating the above arguments, we find
\begin{equation*}
\tilde{v}(t_{0}+T+T,x)\leq T[J\ast\tilde{v}(t_{0}+T,\cdot)](x)\leq T^{2}[J\ast J\ast\tilde{v}(t_{0},\cdot)](x),
\end{equation*}
where we used \eqref{initial-estimate} in the second inequality. Repeating this, we conclude that for any $T>0$ and any $N=1,2,3,\dots$, there holds
\begin{equation}\label{repeating-result}
\tilde{v}(t_{0}+NT,x)\leq T^{N}[J^{N}\ast\tilde{v}(t_{0},\cdot)](x),
\end{equation}
where $J^{N}=\underbrace{J\ast J\ast\cdots\ast J}_{N\,\,\text{times}}$. Note that $J^{N}$ is nonnegative, and if $J$ is compactly supported, then $J^{N}$ is not everywhere positive no matter how large $N$ is. But, since $J$ is nonnegative and positive on some open interval, $J^{N}$ can be positive on any fixed bounded interval if $N$ is large. Moreover, since $J$ is symmetric, so is $J^{N}$.

Now, let $x\in\R$, $z\in\R$ and $h>0$, and let $N:=N(|x-z|,h)$ be large enough so that
\begin{equation*}
\tilde{C}=\tilde{C}(|x-z|,h):=\inf_{y\in[x-z-h,x-z+h]}J^{N}(y)>0.
\end{equation*}
Note that the dependence of $N$ on $x-z$ through $|x-z|$ is due to the symmetry of $J^{N}$. Moreover, the positivity of $\tilde{C}:[0,\infty)\times(0,\infty)\to(0,\infty)$ is uniform on compacts sets, which is because  $N$ can be chosen to be nondecreasing in $|x-z|$ and in $h$.

Then, for $t>t_{0}$, we see from \eqref{repeating-result} with $T=\frac{t-t_{0}}{N}$ that
\begin{equation*}
\begin{split}
\tilde{v}(t,x)&\leq\bigg(\frac{t-t_{0}}{N}\bigg)^{N}\int_{\R}J^{N}(x-y)\tilde{v}(t_{0},y)dy\\
&\leq\bigg(\frac{t-t_{0}}{N}\bigg)^{N}\int_{z-h}^{z+h}J^{N}(x-y)\tilde{v}(t_{0},y)dy\\
&\leq\tilde{C}\bigg(\frac{t-t_{0}}{N}\bigg)^{N}\int_{z-h}^{z+h}\tilde{v}(t_{0},y)dy,
\end{split}
\end{equation*}
since $x-y\in[x-z-h,x-z+h]$ when $y\in[z-h,z+h]$. Going back to $v(t,x)$, we find
\begin{equation*}
u_{1}(t,x;\tau)-u_{2}(t,x;\tau)\leq\tilde{C}e^{-(1+K)(t-t_{0})}\bigg(\frac{t-t_{0}}{N}\bigg)^{N}\int_{z-h}^{z+h}[u_{1}(t_{0},y;\tau)-u_{2}(t_{0},y;\tau)]dy
\end{equation*}
The result then follows with $C=\tilde{C}e^{-(1+K)(t-t_{0})}\big(\frac{t-t_{0}}{N}\big)^{N}$.
\end{proof}

As a simple consequence of Lemma \ref{lem-tech-1234567}, we have

\begin{cor}\label{cor-tech-1234567}
For any $t>t_{0}\geq\tau$, $h>0$ and $z\in\R$, there holds
\begin{equation*}
u_{x}(t,x)\leq C\int_{z-h}^{z+h}u_{x}(t_{0},y)dy,\quad x\in\R,
\end{equation*}
where $C>0$ is as in Lemma \ref{lem-tech-1234567}.
\end{cor}
\begin{proof}
Applying Lemma \ref{lem-tech-1234567} with $u_{1}=u(t,x+\ep)$ and $u_{2}=u(t,x)$, dividing the result by $\ep$ and passing to the limit $\ep\to0^{+}$, we conclude the lemma.
\end{proof}

Now, we prove Theorem \ref{thm-uniform-steepness}.

\begin{proof}[Proof of Theorem \ref{thm-uniform-steepness}]
Recall $X_{\la}(t):=X_{\la}^{\pm}(t)$. By Lemma \ref{lem-bounded-interface-width}, $\sup_{t\in\R}|X(t)-X_{\la}(t)|<\infty$.

Fix any $\la_{0}\in(0,1)$ and set

\begin{equation*}
h_{\la_{0}}:=\max\bigg\{\sup_{t\in\R}|X(t)-X_{\frac{\la_{0}}{2}}(t)|,\sup_{t\in\R}|X(t)-X_{\frac{1+\la_{0}}{2}}(t)|\bigg\}.
\end{equation*}
Then, $h_{\la_{0}}<\infty$ and
\begin{equation}\label{simple-comparison-12345}
X(t)+h_{\la_{0}}\geq X_{\frac{\la_{0}}{2}}(t),\quad X(t)-h_{\la_{0}}\leq X_{\frac{1+\la_{0}}{2}}(t)
\end{equation}
for all $t\in\R$. Now, fix $\tau>0$. For $t\in\R$, we apply Lemma \ref{lem-tech-1234567} with $z=X(t)$ and $h=h_{\la_{0}}$ to see that if $|x-X(t)|\leq M$, then
\begin{equation}\label{aprior-estimate-1234567}
\begin{split}
u_{x}(\tau+t,x)&\leq \tilde{C}(\tau,M,h_{\la_{0}})\int_{X(t)-h_{\la_{0}}}^{X(t)+h_{\la_{0}}}u_{x}(t,y)dy\\
&=\tilde{C}(\tau,M,h_{\la_{0}})[u(t,X(t)+h_{\la_{0}})-u(t,X(t)-h_{\la_{0}})]\\
&\leq\tilde{C}(\tau,M,h_{\la_{0}})[u(t,X_{\frac{\la_{0}}{2}}(t))-u(t,X_{\frac{1+\la_{0}}{2}}(t))]\\
&=-\frac{\tilde{C}(\tau,M,h_{\la_{0}})}{2},
\end{split}
\end{equation}
where we used \eqref{simple-comparison-12345} and the monotonicity in the second inequality, and $\tilde{C}(\tau,M,h_{\la_{0}})=\inf_{K\in[0,M]}C(\tau,K,h_{\la_{0}})$. To apply \eqref{aprior-estimate-1234567}, we see that if $|x-X(t+1)|\leq M$, then
$$|x-X(t)|\leq|x-X(t+1)|+|X(t+1)-X(t)|\leq M+c_{\max},$$
where we used \eqref{interface-eq}. We then apply \eqref{aprior-estimate-1234567} with $M$ replaced by $M+c_{\max}$ and $\tau$ replaced by $1$ to conclude that 
$$
u_{x}(t+1,x)\leq-\frac{1}{2}\inf_{K\in[0,M+c_{\max}]}C(1,K,h_{\la_{0}}). 
$$
Since $t\in\R$ is arbitrary, we arrive at the result.
\end{proof}


\section{Uniform exponential stability of space non-increasing transition fronts}\label{sec-stability}

In this section, we study the stability of space non-increasing transition fronts of \eqref{main-eqn}.
Throughout this section, we assume (H1)-(H3) and assume that
$u(t,x)$ is a transition front of \eqref{main-eqn} with interface location function $X(t)$ and $u_{x}(t,x)\leq 0$.

The main results in this section are stated in the following theorem.

\begin{thm}\label{thm-stability}
\begin{itemize}
\item[\rm(i)]
Let $u_{0}:\R\to[0,1]$ be uniformly continuous and satisfies
\begin{equation*}
\liminf_{x\to-\infty}u_{0}(x)>\theta_{1}\quad\text{and}\quad\limsup_{x\to\infty}u_{0}(x)<\theta_{0},
\end{equation*}
where $\theta_{0}$ and $\theta_{1}$ are as in (H3). Then, for any $t_0\in\R$,  there exist $\xi=\xi(t_0,u_{0})\in\R$, $C=C(u_{0})>0$
 (independent of $t_0$) and $\om_{*}>0$ (independent of $t_0$ and $u_{0}$) such that
\begin{equation*}
\sup_{x\in\R}|u(t,x;t_{0},u_{0})-u(t,x-\xi)|\leq Ce^{-\om_{*}(t-t_{0})}
\end{equation*}
for all $t\geq t_{0}$.

\item[\rm(ii)] Let $\{u_{t_{0}}\}_{t_{0}\in\R}$ be a family of uniformly continuous initial data satisfying
\begin{equation*}
u(t_{0},x-\xi_{0}^{-})-\mu_{0}\leq u_{t_{0}}(x)\leq u(t_{0},x-\xi_{0}^{+})+\mu_{0},\quad x\in\R,\,\,t_{0}\in\R
\end{equation*}
for $\xi_{0}^{\pm}\in\R$ and $\mu_{0}\in(0,\min\{\theta_{0},1-\theta_{1}\})$ being independent of $t_{0}\in\R$. Then, there exist $t_{0}$-independent constants $C>0$ and $\om_{*}>0$, and a family of shifts $\{\xi_{t_{0}}\}_{t_{0}\in\R}\subset\R$ satisfying $\sup_{t_{0}\in\R}|\xi_{t_{0}}|<\infty$ such that
\begin{equation*}
\sup_{x\in\R}|u(t,x;t_{0},u_{t_{0}})-u(t,x-\xi_{t_{0}})|\leq Ce^{-\om_{*}(t-t_{0})}
\end{equation*}
for all $t\geq t_{0}$ and $t_{0}\in\R$.
\end{itemize}
\end{thm}

To prove Theorem \ref{thm-stability}, we first show two lemmas.

\begin{lem}\label{lem-stability}
Let $u_{0}$ be as in Theorem \ref{thm-stability}. Then, for any $t_0\in\R$,
there exist  $\xi_{0}^{\pm}=\xi_{0}^{\pm}(t_0,u_{0})\in\R$, $\mu=\mu(u_{0})>0$ (independent of $t_0$) and $\om=\min\{\beta_{0},\beta_{1}\}>0$ (independent of
$t_0$ and $u_{0}$) such that
\begin{equation}\label{stability-trapping-lem}
u(t,x-\xi^{-}(t))-\mu e^{-\om(t-t_{0})}\leq u(t,x;t_{0},u_{0})\leq u(t,x-\xi^{+}(t))+\mu e^{-\om(t-t_{0})},\quad x\in\R
\end{equation}
for $t\geq t_{0}$, where $\beta_{0}$ and $\beta_{1}$ are as in (H3), and
\begin{equation*}
\xi^{\pm}(t)=\xi_{0}^{\pm}\pm\frac{A\mu}{\om}(1-e^{-\om(t-t_{0})}),\quad t\geq t_{0}
\end{equation*}
for some universal constant $A>0$.
In particular, there holds
\begin{equation*}
u(t,x-\xi^{-})-\mu e^{-\om(t-t_{0})}\leq u(t,x;t_{0},u_{0})\leq u(t,x-\xi^{+})+\mu e^{-\om(t-t_{0})},\quad x\in\R
\end{equation*}
for $t\geq t_{0}$, where $\xi^{\pm}=\xi_{0}^{\pm}\pm\frac{A\mu}{\om}$.
\end{lem}
\begin{proof}
Let $u_{0}$ be as in the statement of Theorem \ref{thm-stability}(i). Let $\mu_{0}^{\pm}=\mu_{0}^{\pm}(u_{0})$ be such that
\begin{equation*}
\theta_{1}<1-\mu_{0}^{-}<\liminf_{x\to-\infty}u_{0}(x)\quad\text{and}\quad\limsup_{x\to\infty}u_{0}(x)<\mu_{0}^{+}<\theta_{0}.
\end{equation*}
Then, for any $t_0\in\R$, we can find  $\xi_{0}^{\pm}=\xi_{0}^{\pm}(t_0,u_{0})$ such that
\begin{equation}\label{comparison-initial-data}
u(t_{0},x-\xi_{0}^{-})-\mu_{0}^{-}\leq u_{0}(x)\leq u(t_{0},x-\xi_{0}^{+})+\mu_{0}^{+},\quad x\in\R.
\end{equation}

To show the lemma, we then construct appropriate sub- and super-solutions and apply comparison principle. We here only prove the first inequality in \eqref{stability-trapping-lem}; the second one can be proven along the same line. To do so, we fix $\om>0$, $A>0$ (to be chosen) and set
\begin{equation*}
u^{-}(t,x)=u(t,x-\xi(t))-\mu_{0}^{-}e^{-\om(t-t_{0})},
\end{equation*}
where $\xi(t)=\xi_{0}^{-}-\frac{A\mu_{0}^{-}}{\om}(1-e^{-\om(t-t_{0})})$. We then compute
\begin{equation*}
\begin{split}
&u^{-}_{t}-[J\ast u^{-}-u^{-}]-f(t,u^{-})\\
&\quad\quad=f(t,u(t,x-\xi(t)))-f(t,u^{-}(t,x))+A\mu_{0}^{-}e^{-\om(t-t_{0})}u_{x}(t,x-\xi(t))+\om\mu_{0}^{-}e^{-\om(t-t_{0})}.
\end{split}
\end{equation*}

Now, we let $M>0$ be so large that
\begin{equation*}
\forall t\in\R,\quad
\begin{cases}
u(t,x)\leq\theta_{0}\quad\text{if}\quad x-X(t)\geq M,\\
u(t,x)\geq\theta_{1}+\mu_{0}^{-}\quad\text{if}\quad x-X(t)\leq-M.
\end{cases}
\end{equation*}
Notice such an $M$ exists due to Lemma \ref{lem-bounded-interface-width}. Then, we see
\begin{itemize}
\item if $x-\xi(t)-X(t)\geq M$, then $u^{-}(t,x)\leq u(t,x-\xi(t))\leq\theta_{0}$, and then by $\rm(H3)$,
\begin{equation*}
\begin{split}
f(t,u(t,x-\xi(t)))-f(t,u^{-}(t,x))&\leq-\beta_{0}[u(t,x-\xi(t))-u^{-}(t,x)]\\
&=-\beta_{0}\mu_{0}^{-}e^{-\om(t-t_{0})}.
\end{split}
\end{equation*}
Since $A\mu_{0}^{-}e^{-\om(t-t_{0})}u_{x}(t,x-\xi(t))\leq0$, we find
\begin{equation*}
u^{-}_{t}-[J\ast u^{-}-u^{-}]-f(t,u^{-})\leq-\beta_{0}\mu_{0}^{-}e^{-\om(t-t_{0})}+\om\mu_{0}^{-}e^{-\om(t-t_{0})}\leq0
\end{equation*}
if $\om\leq\beta_{0}$;

\item if $x-\xi(t)-X(t)\leq-M$, then
\begin{equation*}
u(t,x-\xi(t))\geq u^{-}(t,x)=u(t,x-\xi(t))-\mu_{0}^{-}e^{-\om(t-t_{0})}\geq\theta_{1}+\mu_{0}^{-}-\mu_{0}^{-}=\theta_{1},
\end{equation*}
and then by $\rm(H3)$,
$$f(t,u(t,x-\xi(t)))-f(t,u^{-}(t,x))\leq-\beta_{1}\mu_{0}^{-}e^{-\om(t-t_{0})}.$$ Hence, $u^{-}_{t}-[J\ast u^{-}-u^{-}]-f(t,u^{-})\leq0$ if $\om\leq\beta_{1}$;

\item if $|x-\xi(t)-X(t)|\leq M$, then by Theorem \ref{thm-uniform-steepness},
\begin{equation*}
C_{M}:=\sup_{t\in\R}\sup_{|x-\xi(t)-X(t)|\leq M}u_{x}(t,x-\xi(t))=\sup_{t\in\R}\sup_{|x-X(t)|\leq M}u_{x}(t,x)<0.
\end{equation*}
Since
$$|f(t,u(t,x-\xi(t)))-f(t,u^{-}(t,x))|\leq C_{*}\mu_{0}^{-}e^{-\om(t-t_{0})}$$
for some $C_{*}>0$, we find
\begin{equation*}
u^{-}_{t}-[J\ast u^{-}-u^{-}]-f(t,u^{-})\leq \big(C_{*}\mu_{0}^{-}+A\mu_{0}^{-}C_{M}+\om\mu_{0}^{-}\big)e^{-\om(t-t_{0})}\leq0
\end{equation*}
if $A\geq\frac{C_{*}+\om}{-C_{M}}$.
\end{itemize}
Hence, if we choose $\om=\min\{\beta_{0},\beta_{1}\}$ and $A=\frac{2C_{*}}{-C_{M}}$ (note $\om=\min\{\beta_{0},\beta_{1}\}\leq C_{*}$), we find
\begin{equation*}
u^{-}_{t}\leq J\ast u^{-}-u^{-}+f(t,u^{-}),\quad x\in\R,\quad t>t_{0},
\end{equation*}
that is, $u^{-}(t,x)$ is a sub-solution on $(t_{0},\infty)$. Since
$$
u^{-}(t_{0},x)=u(t_{0},x-\xi_{0}^{-})-\mu_{0}^{-}\leq u_{0}(x)
$$
due to \eqref{comparison-initial-data}, we conclude from comparison principle that
\begin{equation*}
u(t,x-\xi(t))-\mu_{0}^{-}e^{-\om(t-t_{0})}=u^{-}(t,x)\leq u(t,x;t_{0},u_{0}),\quad x\in\R,\quad t\geq t_{0}.
\end{equation*}
Setting $\mu=\max\{\mu_{0}^{-},\mu_{0}^{+}\}$, we complete the proof.
\end{proof}

The proof of Lemma \ref{lem-stability} gives the following

\begin{cor}\label{cor-stability}
Suppose that $\tilde{u}_{0}:\R\to[0,1]$ is uniformly continuous and satisfies
\begin{equation*}
u(t_{0},x-\tilde{\xi}_{0}^{-})-\tilde{\mu}_{0}^{-}\leq \tilde{u}_{0}(x)\leq u(t_{0},x-\tilde{\xi}_{0}^{+})+\tilde{\mu}_{0}^{+},\quad x\in\R
\end{equation*}
for $t_{0}\in\R$, $\tilde{\xi}_{0}^{\pm}\in\R$ and $\tilde{\mu}_{0}^{\pm}>0$ satisfying $\theta_{1}<1-\tilde{\mu}_{0}^{-}$ and $\tilde{\mu}_{0}^{+}<\theta_{0}$, where $\theta_{0}$ and $\theta_{1}$ are as in (H3). Then, there exist $\tilde{\mu}=\max\{\tilde{\mu}_{0}^{-},\tilde{\mu}_{0}^{+}\}>0$ and $\om=\min\{\beta_{0},\beta_{1}\}>0$ such that
\begin{equation*}\label{stability-trapping}
u(t,x-\tilde{\xi}^{-}(t))-\tilde{\mu} e^{-\om(t-t_{0})}\leq u(t,x;t_{0},\tilde{u}_{0})\leq u(t,x-\tilde{\xi}^{+}(t))+\tilde{\mu} e^{-\om(t-t_{0})},\quad x\in\R
\end{equation*}
for $t\geq t_{0}$, where
$$\tilde{\xi}^{\pm}(t)=\tilde{\xi}_{0}^{\pm}\pm\frac{A\tilde{\mu}}{\om}(1-e^{-\om(t-t_{0})}),\quad t\geq t_{0}$$
for some universal constant $A>0$. In particular, we have
\begin{equation*}
u(t,x-\tilde{\xi}^{-})-\tilde{\mu} e^{-\om(t-t_{0})}\leq u(t,x;t_{0},\tilde{u}_{0})\leq u(t,x-\tilde{\xi}^{+})+\tilde{\mu} e^{-\om(t-t_{0})},\quad x\in\R
\end{equation*}
for $t\geq t_{0}$, where $\tilde{\xi}^{\pm}=\tilde{\xi}_{0}^{\pm}\pm\frac{A\tilde{\mu}}{\om}$.
\end{cor}

The next lemma is the key to the proof of Theorem \ref{thm-stability}. We will let $u(t,x;t_{0})$, $t\geq t_{0}$ be a solution with initial data
 $u_0$ at time $t_{0}\in\R$.

\begin{lem}\label{lem-stability-iteration}
There exists $\ep^{*}\in(0,1)$ such that if there holds
\begin{equation}\label{initial-estimate-123}
u(\tau,x-\hat{\xi})-\hat{\de}\leq u(\tau,x;t_{0})\leq u(\tau,x-\hat{\xi}-\hat{h})+\hat{\de},\quad x\in\R
\end{equation}
for some $\tau\geq t_{0}$, $\hat{\xi}\in\R$, $\hat{h}>0$ and $\hat{\de}\in(0,\min\{\theta_{0},1-\theta_{1}\})$, then there exist $\hat{\xi}(t)$, $\hat{h}(t)$ and $\hat{\de}(t)$ satisfying
\begin{equation*}
\begin{split}
&\hat{\xi}(t)\in[\hat{\xi}-\frac{2A\hat{\de}}{\om},\hat{\xi}+\ep^{*}\min\{1,\hat{h}\}]\\
0\leq&\hat{h}(t)\leq\hat{h}-\ep^{*}\min\{1,\hat{h}\}+\frac{4A\hat{\de}}{\om}\\
0\leq&\hat{\de}(t)\leq[\hat{\de}e^{-\om}+C^{*}\ep^{*}\min\{1,\hat{h}\}]e^{-\om(t-\tau-1)}
\end{split}
\end{equation*}
such that
\begin{equation*}
u(t,x-\hat{\xi}(t))-\hat{\de}(t)\leq u(t,x;t_{0})\leq u(t,x-\hat{\xi}(t)-\hat{h}(t))+\hat{\de}(t),\quad x\in\R
\end{equation*}
for $t\geq\tau+1$, where $A>0$ is some universal constant and $C_{*}=\sup_{(t,x)\in\R\times\R}|u_{x}(t,x)|$.
\end{lem}
\begin{proof}
Applying Corollary \ref{cor-stability} to \eqref{initial-estimate-123}, we find
\begin{equation}\label{stability-trapping-1}
u(t,x-\hat{\xi}^{-}(t))-\hat{\de}e^{-\om(t-\tau)}\leq u(t,x;t_{0})\leq u(t,x-\hat{\xi}^{+}(t)-\hat{h})+\hat{\de}e^{-\om(t-\tau)},\quad x\in\R
\end{equation}
for $t\geq\tau$, where $\om=\min\{\beta_{0},\beta_{1}\}$ and $\hat{\xi}^{\pm}(t)=\hat{\xi}\pm\frac{A\hat{\de}}{\om}(1-e^{-\om(t-\tau)})$.

We now modify \eqref{stability-trapping-1} at $t=\tau+1$ to get a new estimate for $u(\tau+1,x;t_{0})$, and then apply Corollary \ref{cor-stability} to this new estimate to conclude the result. To this end, we set
\begin{equation*}
h=\min\{\hat{h},1\}\quad\text{and}\quad C_{\rm steep}=\frac{1}{2}\sup_{t\in\R}\sup_{|x-X(t)|\leq2}u_{x}(t,x).
\end{equation*}
By Theorem \ref{thm-uniform-steepness}, $C_{\rm steep}<0$. Taylor expansion then yields
$$\int_{X(t)-\frac{1}{2}}^{X(t)+\frac{1}{2}}[u(t,x-h)-u(t,x)]dx\geq-2C_{\rm steep}h,\quad\forall t\in\R.$$
In particular, at $t=\tau$, either
\begin{equation}\label{alternative-1}
\int_{X(\tau)-\frac{1}{2}}^{X(\tau)+\frac{1}{2}}[u(\tau,x-h)-u(\tau,x+\hat{\xi};t_{0})]dx\geq-C_{\rm steep}h
\end{equation}
or
\begin{equation}\label{alternative-2}
\int_{X(\tau)-\frac{1}{2}}^{X(\tau)+\frac{1}{2}}[u(\tau,x+\hat{\xi};t_{0})-u(\tau,x)]dx\geq-C_{\rm steep}h
\end{equation}
must be the case.

Suppose first that \eqref{alternative-2} holds. We estimate the following term $$u(\tau+1,x;t_{0})-u(\tau+1,x-\hat{\xi}^{-}(\tau+1)-\ep^{*}h)$$
from below, where $\ep^{*}>0$ is to be chosen. To do so, let $M>0$ and consider two cases: $\rm(i)$ $|x-\hat{\xi}-X(\tau)|\leq M$; $\rm(ii)$ $|x-\hat{\xi}-X(\tau)|\geq M$.

\paragraph{$\rm(i)$ $|x-\hat{\xi}-X(\tau)|\leq M$} In this case, we write
\begin{equation*}
\begin{split}
&u(\tau+1,x;t_{0})-u(\tau+1,x-\hat{\xi}^{-}(\tau+1)-\ep^{*}h)\\
&\quad\quad=[u(\tau+1,x;t_{0})-u(\tau+1,x-\hat{\xi}^{-}(\tau+1))]\\
&\quad\quad\quad+[u(\tau+1,x-\hat{\xi}^{-}(\tau+1))-u(\tau+1,x-\hat{\xi}^{-}(\tau+1)-\ep^{*}h)]\\
&\quad\quad=:{\rm(I)}+{\rm(II)}.
\end{split}
\end{equation*}
For $\rm(I)$, we argue
\begin{equation*}
\begin{split}
{\rm(I)}+\hat{\de}e^{-\om}&=u(\tau+1,x;t_{0})-[u(\tau+1,x-\hat{\xi}+\frac{A\hat{\de}}{\om}(1-e^{-\om}))-\hat{\de}e^{-\om}]\\
&=u(\tau+1,y+\hat{\xi};t_{0})-[u(\tau+1,y+\frac{A\hat{\de}}{\om}(1-e^{-\om}))-\hat{\de}e^{-\om}]\\
&\quad(\text{by}\,\,y=x-\hat{\xi}\in X(\tau)+[-M,M])\\
&=u(\tau+1,y+\hat{\xi};t_{0})-\hat{u}(\tau+1,y)\\
&\quad(\text{where}\,\,\hat{u}(t,y)=u(t,y+\frac{A\hat{\de}}{\om}(1-e^{-\om(t-\tau)}))-\hat{\de}e^{-\om(t-\tau)})\\
&\geq C(M)\int_{X(\tau)-\frac{1}{2}}^{X(\tau)+\frac{1}{2}}[u(\tau,y+\hat{\xi};t_{0})-\hat{u}(\tau,y)]dy\\
&\geq C(M)\int_{X(\tau)-\frac{1}{2}}^{X(\tau)+\frac{1}{2}}[u(\tau,y+\hat{\xi};t_{0})-u(\tau,y)]dy\geq-C(M)C_{\rm steep}h,
\end{split}
\end{equation*}
where the first inequality follows from Lemma \ref{lem-tech-1234567}. In fact, we know $u(t,y+\hat{\xi};t_{0})$ is a solution of $v_{t}=J\ast v-v+f(t,v)$, while $\hat{u}(t,y)$ is a sub-solution by the proof of Lemma \ref{lem-stability}. Moreover, $u(t,y+\hat{\xi};t_{0})\geq\hat{u}(t,y)$ by \eqref{stability-trapping-1}. Then, we apply Lemma \ref{lem-tech-1234567} with $u_{1}=\hat{u}(t,y)$ and $u_{2}=u(t,y+\hat{\xi};t_{0})$ to conclude the inequality. Hence, ${\rm(I)}\geq-\hat{\de}e^{-\om}-C(M)C_{\rm steep}h$.

For $\rm(II)$,  Taylor expansion yields for some $x_{*}\in(0,\ep^{*}h)$
\begin{equation*}
{\rm(II)}=u_{x}(\tau+1,x-\hat{\xi}^{-}(\tau+1)-x_{*})\ep^{*}h\geq-\ep^{*}h\sup_{(t,x)\in\R\times\R}|u_{x}(t,x)|\geq C(M)C_{\rm steep}h
\end{equation*}
if we choose $\ep^{*}=\min\big\{1,\frac{-C(M)C_{\rm steep}}{\sup_{(t,x)\in\R\times\R}|u_{x}(t,x)|}\big\}$. It then follows that
\begin{equation}\label{region-1}
u(\tau+1,x;t_{0})-u(\tau+1,x-\hat{\xi}^{-}(\tau+1)-\ep^{*}h)\geq-\hat{\de}e^{-\om}.
\end{equation}

\paragraph{$\rm(ii)$ $|x-\hat{\xi}-X(\tau)|\geq M$.} In this case, we have
\begin{equation}\label{region-2}
\begin{split}
&u(\tau+1,x;t_{0})-u(\tau+1,x-\hat{\xi}^{-}(\tau+1)-\ep^{*}h)\\
&\quad\quad=[u(\tau+1,x;t_{0})-u(\tau+1,x-\hat{\xi}^{-}(\tau+1))]\\
&\quad\quad\quad+[u(\tau+1,x-\hat{\xi}^{-}(\tau+1))-u(\tau+1,x-\hat{\xi}^{-}(\tau+1)-\ep^{*}h)]\\
&\quad\quad\geq-\hat{\de}e^{-\om}-\ep^{*}h\sup_{(t,x)\in\R\times\R}|u_{x}(t,x)|,
\end{split}
\end{equation}
where we used the first inequality in \eqref{stability-trapping-1} and Taylor expansion.

Hence, by \eqref{region-1}, \eqref{region-2} and the second inequality in \eqref{stability-trapping-1}, we find
\begin{equation}\label{stability-trapping-2}
\begin{split}
&u(\tau+1,x-\hat{\xi}^{-}(\tau+1)-\ep^{*}h)-\hat{\de}e^{-\om}-C^{*}\ep^{*}h\\
&\quad\quad\leq u(\tau+1,x;t_{0})\leq u(\tau+1,x-\hat{\xi}^{+}(\tau+1)-\hat{h})+\hat{\de}e^{-\om},
\end{split}
\end{equation}
where $C^{*}=\sup_{(t,x)\in\R\times\R}|u_{x}(t,x)|$. Taking $\ep_{*}$ smaller, if necessary, so that $\hat{\de}e^{-\om}+C^{*}\ep^{*}h<1-\theta_{1}$, and applying Corollary \ref{cor-stability} to \eqref{stability-trapping-2}, we conclude
\begin{equation}\label{result-1234}
u(t,x-\tilde{\xi}^{-}(t))-\tilde{\de}e^{-\om(t-\tau-1)}\leq u(t,x;t_{0})\leq u(t,x-\tilde{\xi}^{+}(t))+\tilde{\de}e^{-\om(t-\tau-1)}
\end{equation}
for $t\geq\tau+1$, where $\om=\min\{\beta_{0},\beta_{1}\}$, $\tilde{\de}=\max\{\hat{\de}e^{-\om}+C^{*}\ep^{*}h,\hat{\de}e^{-\om}\}=\hat{\de}e^{-\om}+C^{*}\ep^{*}h$ and
\begin{equation*}
\begin{split}
\tilde{\xi}^{-}(t)&=\hat{\xi}^{-}(\tau+1)+\ep^{*}h-\frac{A\hat{\de}}{\om}(1-e^{-\om(t-\tau-1)})\\
&=\hat{\xi}-\frac{2A\hat{\de}}{\om}+\ep^{*}h+\frac{A\hat{\de}}{\om}[e^{-\om}+e^{-\om(t-\tau-1)}],\\
\tilde{\xi}^{+}(t)&=\hat{\xi}^{+}(\tau+1)+\hat{h}+\frac{A\hat{\de}}{\om}(1-e^{-\om(t-\tau-1)})\\
&=\hat{\xi}+\frac{2A\hat{\de}}{\om}+\hat{h}-\frac{A\hat{\de}}{\om}[e^{-\om}+e^{-\om(t-\tau-1)}].
\end{split}
\end{equation*}
Setting
\begin{equation*}
\begin{split}
\hat{\xi}(t)&=\tilde{\xi}^{-}(t)=\hat{\xi}-\frac{2A\hat{\de}}{\om}+\ep^{*}h+\frac{A\hat{\de}}{\om}[e^{-\om}+e^{-\om(t-\tau-1)}],\\
\hat{h}(t)&=\tilde{\xi}^{+}(t)-\tilde{\xi}^{-}(t)=\hat{h}-\ep^{*}h+\frac{4A\hat{\de}}{\om}-\frac{2A\hat{\de}}{\om}[e^{-\om}+e^{-\om(t-\tau-1)}],\\
\hat{\de}(t)&=\tilde{\de}e^{-\om(t-\tau-1)}=[\hat{\de}e^{-\om}+C^{*}\ep^{*}h]e^{-\om(t-\tau-1)},
\end{split}
\end{equation*}
the estimate \eqref{result-1234} can be written as
\begin{equation}\label{result-alternative-2}
u(t,x-\hat{\xi}(t))-\hat{\de}(t)\leq u(t,x;t_{0})\leq u(t,x-\hat{\xi}(t)-\hat{h}(t))+\hat{\de}(t),\quad x\in\R,\,\,t\geq\tau+1.
\end{equation}
Note that \eqref{result-alternative-2} is obtained under the assumption \eqref{alternative-2}.

Now, we assume \eqref{alternative-1} and estimate the following term $$u(\tau+1,x;t_{0})-u(\tau+1,x-\hat{\xi}^{+}(\tau+1)-\hat{h}+\ep^{*}h)$$ from above. Arguing as before and replacing $\hat{h}$ by $h$ at appropriate steps lead to
\begin{equation*}
u(\tau+1,x;t_{0})-u(\tau+1,x-\hat{\xi}^{+}(\tau+1)-\hat{h}+\ep^{*}h)\leq\hat{\de}e^{-\om}+C^{*}\ep^{*}h,
\end{equation*}
where $C^{*}=\sup_{(t,x)\in\R\times\R}|u_{x}(t,x)|$. This, together with the first inequality in \eqref{stability-trapping-1}, yields
\begin{equation}\label{result-alternative-1}
\begin{split}
&u(\tau+1,x-\hat{\xi}^{-}(\tau+1))-\hat{\de}e^{-\om}\\
&\quad\quad\leq u(\tau+1,x;t_{0})\leq u(\tau+1,x-\hat{\xi}^{+}(\tau+1)-\hat{h}+\ep^{*}h)+\hat{\de}e^{-\om}+C^{*}\ep^{*}h.
\end{split}
\end{equation}
Then, applying Corollary \ref{cor-stability} to \eqref{result-alternative-1}, we find \eqref{result-1234} again with
\begin{equation*}
\begin{split}
\hat{\xi}(t)&=\hat{\xi}-\frac{2A\hat{\de}}{\om}+\frac{A\hat{\de}}{\om}[e^{-\om}+e^{-\om(t-\tau-1)}],\\
\hat{h}(t)&=\hat{h}-\ep^{*}h+\frac{4A\hat{\de}}{\om}-\frac{2A\hat{\de}}{\om}[e^{-\om}+e^{-\om(t-\tau-1)}],\\
\hat{\de}(t)&=[\hat{\de}e^{-\om}+C^{*}\ep^{*}h]e^{-\om(t-\tau-1)}.
\end{split}
\end{equation*}
This completes the proof.
\end{proof}

Now, we use the ``squeezing technique" (see e.g. \cite{Ch97,Sh99-1,ShSh14-1}) to prove Theorem \ref{thm-stability}.

\begin{proof}[Proof of Theorem \ref{thm-stability}]
(i) Let $u_{0}$ be the initial data as in the statement of the theorem.
 For any $t_0\in\R$, Lemma \ref{lem-stability} ensures the existence of $\xi^{\pm}=\xi^{\pm}(t_0,u_{0})\in\R$ and $\mu=\mu(u_{0})$
 (independent of $t_0$) such that
\begin{equation*}
u(t,x-\xi^{-})-\mu e^{-\om(t-t_{0})}\leq u(t,x;t_{0},u_{0})\leq u(t,x-\xi^{+})+\mu e^{-\om(t-t_{0})}
\end{equation*}
for $t\geq t_{0}$, where $\om=\min\{\beta_{0},\beta_{1}\}$. Choosing $T_{0}=T_{0}(u_{0})>0$ such that
\begin{equation*}
\de_{0}:=\mu e^{-\om T_{0}}<\de_{*}:=\min\Big\{\theta_{0},1-\theta_{1},\frac{\ep^{*}\om}{8A}\Big\}<1,
\end{equation*}
we find
\begin{equation}\label{iteration-00}
u(t_{0}+T_{0},x-\xi_{0})-\de_{0}\leq u(t_{0}+T_{0},x;t_{0},u_{0})\leq u(t_{0}+T_{0},x-\xi_{0}-h_{0})+\de_{0},
\end{equation}
where $\xi_{0}=\xi^{-}$ and $h_{0}=\xi^{+}-\xi^{-}$. Notice, we may assume, without loss of generality, that $\xi^{+}>\xi^{-}$, so $h_{0}>0$. But, $h_{0}$ depends on $u_{0}$, so we may assume, without loss of generality, that $h_{0}>1$. Let $T>1$ be such that
\begin{equation*}
[e^{-\om}+C^{*}\ep^{*}]e^{-\om (T-1)}\leq\de_{*}:=\min\Big\{\theta_{0},1-\theta_{1},\frac{\ep^{*}\om}{8A}\Big\}.
\end{equation*}
We are going to reduce $h_{0}$.

Applying Lemma \ref{lem-stability-iteration} to \eqref{iteration-00}, we find
\begin{equation}\label{iteration-01}
\begin{split}
&u(t_{0}+T_{0}+T,x-\xi_{1})-\de_{1}\\
&\quad\quad\leq u(t_{0}+T_{0}+T,x;t_{0},u_{0})\leq u(t_{0}+T_{0}+T,x-\xi_{1}-h_{1})+\de_{1},
\end{split}
\end{equation}
where
\begin{equation*}
\begin{split}
&\xi_{1}\in[\xi_{0}-\frac{2A\de_{0}}{\om},\xi_{0}+\ep^{*}\min\{1,h_{0}\}]=[\xi_{0}-\frac{2A\de_{0}}{\om},\xi_{0}+\ep^{*}]\subset[\xi_{0}-\frac{\ep^{*}}{4},\xi_{0}+\ep^{*}],\\
0\leq&h_{1}\leq h_{0}-\ep^{*}\min\{1,h_{0}\}+\frac{4A\de_{0}}{\om}=h_{0}-\ep^{*}+\frac{4A\de_{0}}{\om}\leq h_{0}-\frac{\ep^{*}}{2},\\
0\leq&\de_{1}\leq[\de_{0}e^{-\om}+C^{*}\ep^{*}\min\{1,h_{0}\}]e^{-\om (T-1)}=[\de_{0}e^{-\om}+C^{*}\ep^{*}]e^{-\om (T-1)}\leq\de_{*}.
\end{split}
\end{equation*}
If $h_{1}\leq1$, we stop. Otherwise, we apply Lemma \ref{lem-stability-iteration} to \eqref{iteration-01} to find
\begin{equation}\label{iteration-02}
\begin{split}
&u(t_{0}+T_{0}+2T,x-\xi_{2})-\de_{2}\\
&\quad\quad\leq u(t_{0}+T_{0}+2T,x;t_{0},u_{0})\leq u(t_{0}+T_{0}+2T,x-\xi_{2}-h_{2})+\de_{2},
\end{split}
\end{equation}
where
\begin{equation*}
\begin{split}
&\xi_{2}\in[\xi_{1}-\frac{2A\de_{1}}{\om},\xi_{1}+\ep^{*}\min\{1,h_{1}\}]=[\xi_{1}-\frac{2A\de_{1}}{\om},\xi_{1}+\ep^{*}]\subset[\xi_{1}-\frac{\ep^{*}}{4},\xi_{1}+\ep^{*}],\\
0\leq&h_{2}\leq h_{1}-\ep^{*}\min\{1,h_{1}\}+\frac{4A\de_{1}}{\om}=h_{1}-\ep^{*}+\frac{4A\de_{1}}{\om}\leq h_{0}-2\big(\frac{\ep^{*}}{2}\big),\\
0\leq&\de_{2}\leq[\de_{1}e^{-\om}+C^{*}\ep^{*}\min\{1,h_{1}\}]e^{-\om (T-1)}=[\de_{1}e^{-\om}+C^{*}\ep^{*}]e^{-\om (T-1)}\leq\de_{*}.
\end{split}
\end{equation*}
If $h_{2}\leq1$, we stop. Otherwise, we apply Lemma \ref{lem-stability-iteration} to \eqref{iteration-02}, and repeat this. Suppose $h_{i}>1$ for all $i=0,1,2,\dots n-1$, we then have
\begin{equation}\label{iteration-0n}
\begin{split}
&u(t_{0}+T_{0}+nT,x-\xi_{n})-\de_{n}\\
&\quad\quad\leq u(t_{0}+T_{0}+nT,x;t_{0},u_{0})\leq u(t_{0}+T_{0}+nT,x-\xi_{n}-h_{n})+\de_{n},
\end{split}
\end{equation}
where
\begin{equation*}
\begin{split}
&\xi_{n}\in[\xi_{n-1}-\frac{2A\de_{n-1}}{\om},\xi_{n-1}+\ep^{*}\min\{1,h_{n-1}\}]\subset[\xi_{n-1}-\frac{\ep^{*}}{4},\xi_{n-1}+\ep^{*}],\\
0\leq&h_{n}\leq h_{n-1}-\ep^{*}\min\{1,h_{n-1}\}+\frac{4A\de_{n-1}}{\om}=h_{n-1}-\ep^{*}+\frac{4A\de_{n-1}}{\om}\leq h_{0}-n\big(\frac{\ep^{*}}{2}\big),\\
0\leq&\de_{n}\leq[\de_{n-1}e^{-\om}+C^{*}\ep^{*}\min\{1,h_{n-1}\}]e^{-\om (T-1)}=[\de_{n-1}e^{-\om}+C^{*}\ep^{*}]e^{-\om (T-1)}\leq\de_{*}.
\end{split}
\end{equation*}
Note that since $h_{0}>1$ and $\frac{\ep^{*}}{2}\in(0,1)$, there must exist some $N=N(u_{0})>0$ such that $h_{i}>1$ for $i=0,1,2,\dots,N-1$ and $0<h_{0}-N(\frac{\ep^{*}}{2})\leq1$. In particular, $h_{N}\leq1$. Then, we stop and obtain from \eqref{iteration-0n} that
\begin{equation}\label{new-iteration-00}
u(\tilde{t}_{0},x-\tilde{\xi}_{0})-\tilde{\de}_{0}\leq u(\tilde{t}_{0},x;t_{0},u_{0})\leq u(\tilde{t}_{0},x-\tilde{\xi}_{0}-\tilde{h}_{0})+\tilde{\de}_{0},
\end{equation}
where $\tilde{t}_{0}=t_{0}+T_{0}+NT$, $\tilde{\xi}_{0}=\xi_{N}$, $\tilde{\de}_{0}=\de_{N}\leq\de_{*}$ and $\tilde{h}_{0}=h_{N}\leq1$.

Now, we treat \eqref{new-iteration-00} as the new initial estimate and run the iteration argument again. Let $\tilde{T}>1$ be such that
\begin{equation*}
[e^{-\om}+C^{*}\ep^{*}]e^{-\om(\tilde{T}-1)}\leq\min\Big\{\de_{*},1-\frac{\ep^{*}}{2},\frac{\om}{4A}\frac{\ep_{*}}{2}\big(1-\frac{\ep^{*}}{2}\big)\Big\}.
\end{equation*}
Applying Lemma \ref{lem-stability-iteration} to \eqref{new-iteration-00}, we find
\begin{equation}\label{new-iteration-01}
u(\tilde{t}_{0}+\tilde{T},x-\tilde{\xi}_{1})-\tilde{\de}_{1}\leq u(\tilde{t}_{0}+\tilde{T},x;t_{0},u_{0})\leq u(\tilde{t}_{0}+\tilde{T},x-\tilde{\xi}_{1}-\tilde{h}_{1})+\tilde{\de}_{1},
\end{equation}
where
\begin{equation*}
\begin{split}
&\tilde{\xi}_{1}\in[\tilde{\xi}_{0}-\frac{2A\tilde{\de}_{0}}{\om},\tilde{\xi}_{0}+\ep^{*}\tilde{h}_{0}],\\
0\leq&\tilde{h}_{1}\leq\tilde{h}_{0}-\ep^{*}\tilde{h}_{0}+\frac{4A\tilde{\de}_{0}}{\om}\leq 1-\frac{\ep^{*}}{2},\\
0\leq&\tilde{\de}_{1}\leq[\tilde{\de}_{0}e^{-\om}+C^{*}\ep^{*}\tilde{h}_{0}]e^{-\om (\tilde{T}-1)}\leq\min\Big\{\de_{*},1-\frac{\ep^{*}}{2},\frac{\om}{4A}\frac{\ep_{*}}{2}\big(1-\frac{\ep^{*}}{2}\big)\Big\}.
\end{split}
\end{equation*}
Applying  Lemma \ref{lem-stability-iteration} to \eqref{new-iteration-01}, we find
\begin{equation*}
u(\tilde{t}_{0}+2\tilde{T},x-\tilde{\xi}_{2})-\tilde{\de}_{2}\leq u(\tilde{t}_{0}+2\tilde{T},x;t_{0},u_{0})\leq u(\tilde{t}_{0}+2\tilde{T},x-\tilde{\xi}_{2}-\tilde{h}_{2})+\tilde{\de}_{2},
\end{equation*}
where
\begin{equation*}
\begin{split}
&\tilde{\xi}_{2}\in[\tilde{\xi}_{1}-\frac{2A\tilde{\de}_{1}}{\om},\tilde{\xi}_{1}+\ep^{*}\tilde{h}_{1}],\\
0\leq&\tilde{h}_{2}\leq\tilde{h}_{1}-\ep^{*}\tilde{h}_{1}+\frac{4A\tilde{\de}_{1}}{\om}\leq (1-\frac{\ep^{*}}{2})(1-\ep^{*})+\frac{\ep_{*}}{2}\big(1-\frac{\ep^{*}}{2}\big)=(1-\frac{\ep^{*}}{2})^{2},\\
0\leq&\tilde{\de}_{2}\leq[\tilde{\de}_{1}e^{-\om}+C^{*}\ep^{*}\tilde{h}_{1}]e^{-\om (\tilde{T}-1)}\leq(1-\frac{\ep^{*}}{2})\times\min\Big\{\de_{*},1-\frac{\ep^{*}}{2},\frac{\om}{4A}\frac{\ep_{*}}{2}\big(1-\frac{\ep^{*}}{2}\big)\Big\}.
\end{split}
\end{equation*}
Applying Lemma \ref{lem-stability-iteration} repeatedly, we find for $n\geq3$
\begin{equation}\label{estimate-n-th-101010}
u(\tilde{t}_{0}+n\tilde{T},x-\tilde{\xi}_{n})-\tilde{\de}_{n}\leq u(\tilde{t}_{0}+n\tilde{T},x;t_{0},u_{0})\leq u(\tilde{t}_{0}+n\tilde{T},x-\tilde{\xi}_{n}-\tilde{h}_{n})+\tilde{\de}_{n},
\end{equation}
where
\begin{equation*}
\begin{split}
&\tilde{\xi}_{n}\in[\tilde{\xi}_{n-1}-\frac{2A\tilde{\de}_{n-1}}{\om},\tilde{\xi}_{n-1}+\ep^{*}\tilde{h}_{n-1}],\\
0\leq&\tilde{h}_{n}\leq\tilde{h}_{n-1}-\ep^{*}\tilde{h}_{n-1}+\frac{4A\tilde{\de}_{n-1}}{\om}\leq (1-\frac{\ep^{*}}{2})^{n-1}(1-\ep^{*})+\frac{\ep_{*}}{2}\big(1-\frac{\ep^{*}}{2}\big)^{n-1}=(1-\frac{\ep^{*}}{2})^{n},\\
0\leq&\tilde{\de}_{n}\leq[\tilde{\de}_{n-1}e^{-\om}+C^{*}\ep^{*}\tilde{h}_{n-1}]e^{-\om (\tilde{T}-1)}\leq(1-\frac{\ep^{*}}{2})^{n-1}\times\min\Big\{\de_{*},1-\frac{\ep^{*}}{2},\frac{\om}{4A}\frac{\ep^{*}}{2}\big(1-\frac{\ep^{*}}{2}\big)\Big\}.
\end{split}
\end{equation*}

The result then follows readily. In fact, applying Corollary \ref{cor-stability} to \eqref{estimate-n-th-101010} for $n\geq0$, we find, in particular,
$$
u(t,x-\tilde{\xi}_{n}+\frac{A}{\om}\tilde{\de}_{n})-\tilde{\de}_{n}\leq u(t,x;t_{0},u_{0})\leq u(t,x-\tilde{\xi}_{n}-\tilde{h}_{n}-\frac{A}{\om}\tilde{\de}_{n})+\tilde{\de}_{n}
$$
for all $t\geq \tilde{t}_{0}+n\tilde{T}$. Therefore, for $t\in[\tilde{t}_{0}+n\tilde{T},\tilde{t}_{0}+(n+1)\tilde{T})$, setting $\tilde{\de}(t)=\tilde{\de}_{n}$, $\tilde{\xi}(t)=\tilde{\de}_{n}-\frac{A}{\om}\tilde{\de}_{n}$ and $\tilde{h}(t)=\tilde{h}_{n}+\frac{2A}{\om}\tilde{\de}_{n}$, we arrive at
$$
u(t,x-\tilde{\xi}(t))-\tilde{\de}(t)\leq u(t,x;t_{0},u_{0})\leq u(t,x-\tilde{\xi}(t)-\tilde{h}(t))+\tilde{\de}(t).
$$
Hence, we find
$$
u(t,x-\tilde{\xi}(t))-\tilde{\de}(t)\leq u(t,x;t_{0},u_{0})\leq u(t,x-\tilde{\xi}(t)-\tilde{h}(t))+\tilde{\de}(t),\quad t\geq \tilde{t}_{0},
$$
which then yields the result, since $\tilde{\de}(t)\to0$, $\tilde{\xi}(t)\to\tilde{\xi}(\infty)$ and $\tilde{h}(t)\to0$ exponentially as $t\to\infty$.

We remark that the dependence of $C$ on $u_{0}$ in the statement of the theorem is due to the dependence of $T_{0}$ on $u_{0}$.

(ii) By Corollary \ref{cor-stability}, we see
\begin{equation*}
u(t,x-\xi^{-})-\mu_{0}e^{-\om(t-t_{0})}\leq u(t,x;t_{0},u_{t_{0}})\leq u(t,x-\xi^{+})+\mu_{0}e^{-\om(t-t_{0})},\quad x\in\R
\end{equation*}
for all $t\geq t_{0}$ and $t_{0}\in\R$, where $\om=\min\{\beta_{0},\beta_{1}\}$ and $\xi^{\pm}=\xi^{\pm}_{0}\pm\frac{A\mu_{0}}{\om}$.
Then, by the arguments as in (i), there exist $t_{0}$-independent constants $C>0$ and $\om_{*}>0$, and a family of shifts $\{\xi_{t_{0}}\}_{t_{0}\in\R}\subset\R$ satisfying $\sup_{t_{0}\in\R}|\xi_{t_{0}}|<\infty$ such that
\begin{equation*}
\sup_{x\in\R}|u(t,x;t_{0},u_{t_{0}})-u(t,x-\xi_{t_{0}})|\leq Ce^{-\om_{*}(t-t_{0})}
\end{equation*}
for all $t\geq t_{0}$ and $t_{0}\in\R$.
\end{proof}


\section{Exponential decaying estimates of space non-increasing transition fronts}\label{sec-decaying-estimate}

In this section, we prove exponential decaying estimates of space non-increasing transition fronts of \eqref{main-eqn}. Throughout this section, we assume (H1)-(H3) and assume that $u(t,x)$ is a transition front of \eqref{main-eqn} with interface location functions $X(t)$ and $X_{\la}(t)$ and $u_x(t,x)\le 0$.

The main results in this section are stated in the following theorem.

\begin{thm}\label{thm-expo-decaying}
There exist $c^{\pm}>0$ and $h^{\pm}>0$ such that
\begin{equation*}
u(t,x)\leq e^{-c^{+}(x-X(t)-h^{+})}\quad\text{and}\quad1-u(t,x)\leq e^{c^{-}(x-X(t)+h^{-})}
\end{equation*}
for all $(t,x)\in\R\times\R$. In particular, for any $\la\in(0,1)$, there exist $h_{\la}^{\pm}>0$ such that
\begin{equation*}
u(t,x)\leq e^{-c^{+}(x-X_{\la}(t)-h_{\la}^{+})}\quad\text{and}\quad1-u(t,x)\leq e^{c^{-}(x-X_{\la}(t)+h_{\la}^{-})}
\end{equation*}
for all $(t,x)\in\R\times\R$.
\end{thm}

To prove Theorem \ref{thm-expo-decaying}, we first prove several lemmas. Let $\theta_{2}\in(0,\min\{\frac{1}{4},\theta_{0},1-\theta_{1}\})$ be small and $h>0$, and define $u_{0}^{\pm}:\R\to[0,1]$ to be smooth and non-increasing functions satisfying
\begin{equation}\label{two-special-functions}
u_{0}^{+}(x)=\begin{cases}
1-\theta_{2},\quad&x\leq-h,\\
0,\quad&x\geq0,
\end{cases}
\quad \text{and}\quad
u_{0}^{-}(x)=\begin{cases}
1,\quad&x\leq0,\\
\theta_{2},\quad&x\geq h.
\end{cases}
\end{equation}
Moreover, we can make $u_{0}^{+}$ decreasing on $(-h,0)$ and $u_{0}^{-}$ decreasing on $(0,h)$. For $t_{0}\in\R$, we define
\begin{equation*}
\begin{split}
u^{+}(t,x;t_{0})&:=u(t,x;t_{0},u_{0}^{+}(\cdot-X_{1-\theta_{2}}(t_{0}))),\\
u^{-}(t,x;t_{0})&:=u(t,x;t_{0},u_{0}^{-}(\cdot-X_{\theta_{2}}(t_{0})))
\end{split}
\end{equation*}
for $t\geq t_{0}$.

\begin{lem}\label{lem-properties-auxiliary-fun}
$u^{\pm}(t,x;t_{0})$ satisfy the following properties:
\begin{itemize}
\item[\rm(i)] $u^{\pm}(t,x;t_{0})$ are decreasing in $x$ for any $t>t_{0}$;

\item[\rm(ii)] for any $t>t_{0}$, we have
\begin{equation*}
\begin{split}
&\lim_{x\to-\infty}u^{+}(t,x;t_{0})>1-\theta_{2},\quad \lim_{x\to\infty}u^{+}(t,x;t_{0})=0,\\
&\lim_{x\to-\infty}u^{-}(t,x;t_{0})=1\quad\text{and}\quad \lim_{x\to\infty}u^{-}(t,x;t_{0})<\theta_{2}.
\end{split}
\end{equation*}
\end{itemize}
\end{lem}
\begin{proof}
(i) It follows from the fact that $u^{\pm}_{0}$ are non-increasing and the ``Moreover" part in Proposition \ref{prop-app-comparison}(iii) or Proposition \ref{prop-app-cp-nonlinear}(ii).

(ii) By (i), the limits $\lim_{x\to\pm\infty}u^{+}(t,x;t_{0})$ are well-defined. We show $\lim_{x\to\infty}u^{+}(t,x;t_{0})=0$. Let $\phi_{\tilde{B}}(x-c_{\tilde{B}}t)$ be a traveling wave of $u_{t}=J\ast u-u+f_{\tilde{B}}(u)$ such that $\phi_{\tilde{B}}(-\infty)=1$ and $\phi_{\tilde{B}}(\infty)=0$. By the definition of $u^{+}_{0}$, we can find a shift $x_{1}\gg1$ such that
\begin{equation*}
u_{0}^{+}(\cdot-X_{1-\theta_{2}}(t_{0}))\leq\phi_{\tilde{B}}(\cdot-x_{1}-c_{\tilde{B}}t_{0}).
\end{equation*}
It then follows from comparison principle that $u^{+}(t,\cdot;t_{0})\leq\phi_{\tilde{B}}(\cdot-x_{1}-c_{\tilde{B}}t)$ for any $t>0$. From which, we conclude $\lim_{x\to\infty}u^{+}(t,x;t_{0})=0$.

We show $\lim_{x\to-\infty}u^{+}(t,x;t_{0})>1-\theta_{2}$. Note that $u^{+}(t,x;t_{0})$ satisfies
\begin{equation}\label{eqn-aux-limit-17890}
u^{+}_{t}(t,x;t_{0})=\int_{\R}J(x-y)u^{+}(t,y;t_{0})dy-u^{+}(t,x;t_{0})+f(t,u^{+}(t,x;t_{0})),\quad t>t_{0},
\end{equation}
and
\begin{equation*}
\begin{split}
u^{+}_{tt}(t,x;t_{0})&=\int_{\R}J(x-y)u^{+}_{t}(t,y;t_{0})dy-u_{t}^{+}(t,x;t_{0})\\
&\quad+f_{t}(t,u^{+}(t,x;t_{0}))+f_{u}(t,u^{+}(t,x;t_{0}))u^{+}_{t}(t,x;t_{0}),\quad t>t_{0}.
\end{split}
\end{equation*}
Pick an arbitrary sequence $\{x_{n}\}$ with $x_{n}\to-\infty$ as $n\to\infty$. We see that there is an $M>0$ such that
\begin{equation*}
\max\big\{|u_{t}^{+}(t,x_{n};t_{0})|,|u_{t}^{+}(t,x_{n};t_{0})|\big\}\leq M,\quad t>t_{0},\quad n\geq1.
\end{equation*}
Since $u^{+}(t,x_{n};t_{0})\in[0,1]$ for all $t\geq t_{0}$ and $n\geq1$, we conclude from the Arzel\`{a}-Ascoli theorem that there exists a continuous function $w:[t_{0},\infty)\to[0,1]$, differentiable on $(t_{0},\infty)$ such that
\begin{equation*}
\begin{split}
&u^{+}(t,x_{n};t_{0})\to w(t)\,\,\text{locally uniformly in $t\in[t_{0},\infty)$ as}\,\, n\to\infty,\quad \text{and}\\
&u_{t}^{+}(t,x_{n};t_{0})\to w_{t}(t)\,\,\text{locally uniformly in $t\in(t_{0},\infty)$ as}\,\, n\to\infty.
\end{split}
\end{equation*}
As a consequence, letting $x\to-\infty$ along the sequence $\{x_{n}\}$ in \eqref{eqn-aux-limit-17890}, we find that $w(t)$ is the unique solution of
\begin{equation*}
\begin{cases}
w_{t}(t)=f(t,w(t)),\quad t>t_{0}\\
w(t_{0})=1-\theta_{2}.
\end{cases}
\end{equation*}
Now, comparing $f(t,u)$ with $f_{B}(u)$, we conclude from the comparison principle for ODEs that $w(t)>1-\theta_{2}$ for all $t>t_{0}$. But the monotonicity of $u^{+}(t,x;t_{0})$ in $x$ from (i) yields
\begin{equation*}
\lim_{x\to-\infty}u^{+}(t,x;t_{0})=\lim_{n\to\infty}u^{+}(t,x_{n};t_{0})=w(t)>1-\theta_{2},\quad t>t_{0}.
\end{equation*}

The limits $\lim_{x\to-\infty}u^{-}(t,x;t_{0})=1$ and $\lim_{x\to\infty}u^{-}(t,x;t_{0})<\theta_{2}$ follow from similar arguments, and therefore, we omit the proof.
\end{proof}

By Lemma \ref{lem-properties-auxiliary-fun}, for any $\la\in(\theta_{2},1-\theta_{2})$, the interface locations $X_{\la}^{\pm}(t;t_{0})\in\R$ such that 
$$
u^{\pm}(t,X_{\la}^{\pm}(t;t_{0});t_{0})=\la
$$ 
are well-defined for all $t\geq t_{0}$.

The first lemma gives the uniform boundedness of the gap between the interface locations of $u^{\pm}(t,x;t_{0})$ and $u(t,x)$.

\begin{lem}\label{lem-bounded-interface-width-1}
For any $\la\in(\theta_{2},1-\theta_{2})$, there hold
\begin{equation*}
\sup_{t_{0}\in\R}\sup_{t\geq t_{0}}|X^{\pm}_{\la}(t;t_{0})-X(t)|<\infty.
\end{equation*}
\end{lem}
\begin{proof}
Let $\la\in(\theta_{2},1-\theta_{2})$. By the definition of $u_{0}^{+}$, we see that $u_{0}^{+}(x-X_{1-\theta_{2}}(t_{0}))\leq u(t_{0},x)$ for $x\in\R$. Comparison principle then yields $u^{+}(t,x;t_{0})\leq u(t,x)$ for $x\in\R$ and $t\geq t_{0}$. In particular, $X^{+}_{\la}(t;t_{0})\leq X_{\la}(t)$ for all $t\geq t_{0}$.

Moreover, we readily check that
$$u_{0}^{+}(x-X_{\theta_{2}}(t_{0})-h)+\theta_{2}\geq u(t_{0},x),$$
which is equivalent to
\begin{equation*}
u(t_{0},x+X_{\theta_{2}}(t_{0})+h-X_{1-\theta_{2}}(t_{0}))-\theta_{2}\leq u_{0}^{+}(x-X_{1-\theta_{2}}(t_{0}))=u^{+}(t_{0},x;t_{0}).
\end{equation*}
Setting $L:=\sup_{t_{0}\in\R}|X_{\theta_{2}}(t_{0})+h-X_{1-\theta_{2}}(t_{0})|<\infty$, we see from the monotonicity of $u(t,x)$ in $x$ that $$u(t_{0},x-(-L))-\theta_{2}\leq u^{+}(t_{0},x;t_{0}).$$
Since $L$ and $\theta_{2}$ are $t_{0}$-independent, we apply Theorem \ref{thm-stability} to conclude that
\begin{equation*}
u(t,x-(-L-\frac{A\theta_{2}}{\om}))-\theta_{2}\leq u(t,x-(-L-\frac{A\theta_{2}}{\om}))-\theta_{2}e^{-\om(t-t_{0})}\leq u^{+}(t,x;t_{0}),\quad x\in\R
\end{equation*}
for all $t\geq t_{0}$ and $t_{0}\in\R$. Setting $x=-L-\frac{A\theta_{2}}{\om}+X_{\la+\theta_{2}}(t)$, we find
$$\la\leq u^{+}(t,-L-\frac{A\theta_{2}}{\om}+X_{\la+\theta_{2}}(t);t_{0}),$$
which implies by monotonicity that
$$X_{\la}^{+}(t;t_{0})\geq-L-\frac{A\theta_{2}}{\om}+X_{\la+\theta_{2}}(t)\quad\text{for all}\quad t\geq t_{0}.$$

Hence, we have shown that
\begin{equation*}
X^{+}_{\la}(t;t_{0})\leq X_{\la}(t)\quad\text{and}\quad X_{\la}^{+}(t;t_{0})\geq-L-\frac{A\theta_{2}}{\om}+X_{\la+\theta_{2}}(t)
\end{equation*}
for all $t\geq t_{0}$ and $t_{0}\in\R$. Since $\sup_{t\in\R}|X_{\la}(t)-X_{\la+\theta_{2}}(t)|<\infty$, we arrive at
\begin{equation*}
\sup_{t_{0}\in\R}\sup_{t\geq t_{0}}|X^{+}_{\la}(t;t_{0})-X_{\la}(t)|<\infty,
\end{equation*}
which is clearly equivalent to $\sup_{t_{0}\in\R}\sup_{t\geq t_{0}}|X^{+}_{\la}(t;t_{0})-X(t)|<\infty$.

The another result $\sup_{t_{0}\in\R}\sup_{t\geq t_{0}}|X^{-}_{\la}(t;t_{0})-X(t)|<\infty$ follows along the same line.
\end{proof}

Next, we prove the uniform exponential decaying estimates of $u^{\pm}(t,x;t_{0})$.

\begin{lem}\label{lem-expo-decaying-sepcial-data}
There exist $c^{\pm}>0$ and $h^{\pm}>0$ such that
\begin{equation*}\label{expo-decaying-sepcial-data}
u^{+}(t,x;t_{0})\leq e^{-c^{+}(x-X(t)-h^{+})}\quad\text{and}\quad u^{-}(t,x;t_{0})\geq 1-e^{c^{-}(x-X(t)+h^{-})}
\end{equation*}
for all $x\in\R$, $t\geq t_{0}$ and $t_{0}\in\R$.
\end{lem}
\begin{proof}
We prove the first estimate; the second one can be proven in a similar way. Note first that $f(t,u)\leq-\beta_{0}u$ for $u\in[0,\theta_{0}]$. Let $h:=\sup_{t\geq t_{0}}|X_{\theta_{0}}^{+}(t;t_{0})-X(t)|<\infty$ by Lemma \ref{lem-bounded-interface-width-1}, since $\theta_{0}\in(\theta_{2},1-\theta_{2})$. We consider
\begin{equation*}
N[u]=u_{t}-[J\ast u-u]+\beta_{0}u.
\end{equation*}

Since $u^{+}(t,x;t_{0})\leq\theta_{0}$ for $x\geq X^{+}_{\theta_{0}}(t;t_{0})$, we find
$$N[u^{+}]=\beta_{0}u^{+}+f(t,u^{+})\leq0\quad \text{for}\quad x\geq X^{+}_{\theta_{0}}(t;t_{0}).$$
In particular, $N[u^{+}]\leq0$ for $x\geq X(t)+h$.

Now, let $c>0$. We see
\begin{equation*}
N[e^{-c(x-X(t)-h)}]=\bigg[c\dot{X}(t)-\int_{\R}J(y)e^{cy}dy+1+\beta_{0}\bigg]e^{-c(x-X(t)-h)}.
\end{equation*}
Since $\dot{X}(t)\geq c_{\min}>0$ by \eqref{interface-eq} and $\int_{\R}J(y)e^{cy}dy\to1$ as $c\to0$, we can find some $c_{*}>0$ such that $N[e^{-c_{*}(x-X(t)-h)}]\geq0$. Thus, we have
\begin{itemize}
\item $N[u^{+}(t,x;t_{0})]\leq0\leq N[e^{-c_{*}(x-X(t)-h)}]$ for $x\geq X(t)+h$ and $t>t_{0}$,
\item $u^{+}(t,x;t_{0})<1\leq e^{-c_{*}(x-X(t)-h)}$ for $x\leq X(t)+h$ and $t>t_{0}$,
\item $u^{+}(t_{0},x;t_{0})=u_{0}^{+}(x-X_{1-\theta_{2}}(t_{0}))\leq e^{-c_{*}(x-X(t_{0})-h)}$ for $x\in\R$.
\end{itemize}
We then conclude from Proposition \ref{prop-app-comparison}(i) that $u^{+}(t,x;t_{0})\leq e^{-c_{*}(x-X(t)-h)}$ for all $x\in\R$, $t\geq t_{0}$ and $t_{0}\in\R$. This completes the proof.
\end{proof}

We also need the uniform-in-$t_{0}$ exponential convergence of $u^{\pm}(t,x;t_{0})$ to $u(t,x)$.

\begin{lem}\label{lem-uniform-stability-special-data}
There exist $t_{0}$-independent constants $C>0$ and $\om_{*}>0$, and two families of shifts $\{\xi^{\pm}_{t_{0}}\}_{t_{0}\in\R}\subset\R$ satisfying $\sup_{t_{0}\in\R}|\xi^{\pm}_{t_{0}}|<\infty$ such that
\begin{equation*}
\sup_{x\in\R}|u^{\pm}(t,x;t_{0})-u(t,x-\xi^{\pm}_{t_{0}})|\leq Ce^{-\om_{*}(t-t_{0})}
\end{equation*}
for all $t\geq t_{0}$ and $t_{0}\in\R$.
\end{lem}
\begin{proof}
Let $C_{2}=\sup_{t\in\R}|X_{\theta_{2}}(t)-X_{1-\theta_{2}}(t)|<\infty$. Then, it is easy to see that for any $t_{0}\in\R$
\begin{equation*}
\begin{split}
&u(t_{0},x+C_{2}+h)-\theta_{2}\leq u_{0}^{+}(x-X_{1-\theta_{2}}(t_{0}))\leq u(t_{0},x)+\ep_{0},\quad x\in\R\\
&u(t_{0},x)-\ep_{0}\leq u_{0}^{-}(x-X_{\theta_{2}}(t_{0}))\leq u(t_{0},x-C_{2}-h)+\theta_{2},\quad x\in\R
\end{split}
\end{equation*}
for arbitrary fixed $\ep_{0}\in(0,\min\{\frac{1}{4},\theta_{0},1-\theta_{1}\})$, that is,
\begin{equation*}
\begin{split}
&u(t_{0},x+C_{2}+h)-\mu_{0}\leq u^{+}(t_{0},x;t_{0})\leq u(t_{0},x)+\mu_{0},\quad x\in\R\\
&u(t_{0},x)-\mu_{0}\leq u^{-}(t_{0},x;t_{0})\leq u(t_{0},x-C_{2}-h)+\mu_{0},\quad x\in\R,
\end{split}
\end{equation*}
where $\mu_{0}=\max\{\theta_{2},\ep_{0}\}$. Since $C_{2}$, $h$ and $\mu_{0}$ are independent of $t_{0}\in\R$, we apply
Theorem \ref{thm-stability} to conclude the result.
\end{proof}

Finally, we prove Theorem \ref{thm-expo-decaying}.

\begin{proof}[Proof of Theorem \ref{thm-expo-decaying}]
By Lemma \ref{lem-expo-decaying-sepcial-data} and Lemma \ref{lem-uniform-stability-special-data}, we have
\begin{equation*}
u(t,x-\xi^{+}_{t_{0}})\leq u^{+}(t,x;t_{0})+Ce^{-\om_{*}(t-t_{0})}\leq e^{-c^{+}(x-X(t)-h^{+})}+Ce^{-\om_{*}(t-t_{0})}
\end{equation*}
for all $x\in\R$ and $t\geq t_{0}$. Since $\sup_{t_{0}\in\R}|\xi_{t_{0}}^{+}|<\infty$, there exists $\xi^{+}\in\R$ such that $\xi_{t_{0}}^{+}\to\xi^{+}$ as $t_{0}\to-\infty$ along some subsequence. Thus, for any $(t,x)\in\R\times\R$, letting $t_{0}\to-\infty$ along this subsequence, we find $u(t,x-\xi^{+})\leq e^{-c^{+}(x-X(t)-h^{+})}$. The lower bound for $u(t,x)$ follows similarly. The ``in particular" part then is a simple consequence of the fact that $\sup_{t\in\R}|X_{\la}(t)-X(t)|<\infty$ for any $\la\in(0,1)$.
\end{proof}



\section{Uniqueness and monotonicity of transition fronts}\label{sec-uniqueness}

In this section, we study the uniqueness and monotonicity of transition fronts of \eqref{main-eqn} under the assumptions Hypothesis (H1)-(H3) and
the assumption that \eqref{main-eqn} has a space non-increasing transition front $u(t,x)$.

Let $v(t,x)$ be an arbitrary transition front (not necessarily non-increasing in space), and $u(t,x)$ be an arbitrary space non-increasing transition front of \eqref{main-eqn}.
Let $Y(t)$, $Y_{\la}^{\pm}(t)$ be the interface location functions of $v(t,x)$, and $X(t)$, $X_{\la}(t)=X_{\la}^{\pm}(t)$ be the interface location functions of $u(t,x)$. By Proposition \ref{prop-regularity-of-tf}, we may assume that both $X(t)$ and $Y(t)$ are continuously differentiable and satisfy \eqref{interface-eq}. By Corollary \ref{cor-diff-interface-mono}, $X_{\la}(t)$ is continuously differentiable. But, $Y_{\la}^{\pm}(t)$ may have a jump.

We prove

\begin{thm}\label{thm-uniqeness}
There exists some $\xi\in\R$ such that $v(t,x)=u(t,x+\xi)$ for all $(t,x)\in\R\times\R$. In particular, $v(t,x)$ is non-increasing in $x$.
\end{thm}

To show Theorem \ref{thm-uniqeness}, we first prove the following lemma.

\begin{lem}\label{lem-bounded-interface-width-uniqueness-general}
There holds $\sup_{t\in\R}|X(t)-Y(t)|<\infty$.
\end{lem}
\begin{proof}
Since $\sup_{t\in\R}|X_{\frac{1}{2}}(t)-X(t)|<\infty$, it suffices to show: $\rm(i)$ $\sup_{t\geq0}|Y(t)-X_{\frac{1}{2}}(t)|<\infty$; $\rm(ii)$ $\sup_{t\leq0}|Y(t)-X_{\frac{1}{2}}(t)|<\infty$.

$\rm(i)$ Let $\mu\in(0,\min\{\frac{1}{4},\theta_{0},1-\theta_{1}\})$ be small. We first see that
\begin{equation}\label{initial-estimate-uniqueness-general}
u(0,x-Y_{1-\mu}^{-}(0)+X_{\mu}(0))-\mu\leq v(0,x)\leq u(0,x-Y_{\mu}^{+}(0)+X_{1-\mu}(0))+\mu,\quad x\in\R.
\end{equation}
In fact, if $x\geq Y_{1-\mu}^{-}(0)$, then by the monotonicity of $u(t,x)$ in $x$, we have
\begin{equation*}
u(0,x-Y_{1-\mu}^{-}(0)+X_{\mu}(0))-\mu\leq u(0,X_{\mu}(0))-\mu=0<v(0,x).
\end{equation*}
If $x<Y_{1-\mu}^{-}(0)$, then
$$v(0,x)\geq 1-\mu>u(0,x-Y_{1-\mu}^{-}(0)+X_{\mu}(0))-\mu.$$
This proves the first inequality. The second one is checked similarly.

Setting $\xi^{-}_{0}=Y_{1-\mu}^{-}(0)-X_{\mu}(0)$ and $\xi_{0}^{+}=Y_{\mu}^{+}(0)-X_{1-\mu}(0)$ in \eqref{initial-estimate-uniqueness-general}, and then, applying Corollary \ref{cor-stability} to \eqref{initial-estimate-uniqueness-general}, we find
\begin{equation}\label{estimate-uniqueness-general}
u(t,x-\xi^{-})-\mu\leq v(t,x)\leq u(t,x-\xi^{+})+\mu,\quad x\in\R
\end{equation}
for all $t\geq0$, where $\xi^{\pm}=\xi^{\pm}_{0}\pm\frac{A\mu}{\om}$.
It then follows from the first inequality in \eqref{estimate-uniqueness-general} and the monotonicity of $u(t,x)$ in $x$ that
\begin{equation*}
\frac{1}{2}-\mu=u(t,X_{\frac{1}{2}}(t))-\mu<u(t,x-\xi^{-})-\mu\leq v(t,x)\quad\text{for all}\quad x<\xi^{-}+X_{\frac{1}{2}}(t),
\end{equation*}
which implies that $\xi^{-}+X_{\frac{1}{2}}(t)\leq Y^{-}_{\frac{1}{2}-\mu}(t)$ for $t\geq0$. Similarly, the second inequality in \eqref{estimate-uniqueness-general} and the monotonicity of $u(t,x)$ in $x$ implies that
\begin{equation*}
v(t,x)\leq u(t,x-\xi^{+})+\mu<u(t,X_{\frac{1}{2}}(t))+\mu=\frac{1}{2}+\mu\quad\text{for all}\quad x>\xi^{+}+X_{\frac{1}{2}}(t),
\end{equation*}
which leads to $Y_{\frac{1}{2}+\mu}^{+}(t)\leq\xi^{+}+X_{\frac{1}{2}}(t)$ for $t\geq0$. Since $\sup_{t\in\R}|Y^{-}_{\frac{1}{2}-\mu}(t)-Y(t)|<\infty$ and $\sup_{t\in\R}|Y(t)-Y_{\frac{1}{2}+\mu}^{+}(t)|<\infty$ by Lemma \ref{lem-bounded-interface-width}, we conclude that $\sup_{t\geq0}|X_{\frac{1}{2}}(t)-Y(t)|<\infty$.

$\rm(ii)$ Suppose on the contrary that $\sup_{t\leq0}|Y(t)-X_{\frac{1}{2}}(t)|=\infty$. Since both $Y(t)$ and $X_{\frac{1}{2}}(t)$ are continuous, there exists a sequence $t_{n}\to-\infty$ as $n\to\infty$ such that either $Y(t_{n})-X_{\frac{1}{2}}(t_{n})\to\infty$ or $Y(t_{n})-X_{\frac{1}{2}}(t_{n})\to-\infty$ as $n\to\infty$.

Suppose first that $Y(t_{n})-X_{\frac{1}{2}}(t_{n})\to\infty$ as $n\to\infty$. Since $\sup_{t\in\R}|Y(t)-Y^{-}_{\frac{1}{2}}(t)|<\infty$, we in particular have $Y^{-}_{\frac{1}{2}}(t_{n})-X_{\frac{1}{2}}(t_{n})\to\infty$ as $n\to\infty$. Then, for any $\mu>0$ and $\xi_{0}\in\R$, we can find an $N=N(\mu,\xi_{0})>0$ such that $t_{N}<0$ and $u(t_{N},x-\xi_{0})-\mu\leq v(t_{N},x)$ for $x\in\R$. We then apply Corollary \ref{cor-stability} to conclude that
\begin{equation*}
u(t,x-\xi_{0}+\frac{A\mu}{\om})-\mu\leq v(t,x),\quad x\in\R,\quad t\geq t_{N}.
\end{equation*}
Then, setting $t=0$ in the above estimate, we find from the monotonicity of $u(t,x)$ in $x$ that
\begin{equation*}
\frac{1}{2}-\mu=u(0,X_{\frac{1}{2}}(0))-\mu<u(0,x-\xi_{0}+\frac{A\mu}{\om})-\mu\leq v(0,x),\quad \forall x<\xi_{0}-\frac{A\mu}{\om}+X_{\frac{1}{2}}(0),
\end{equation*}
which implies that $\xi_{0}-\frac{A\mu}{\om}+X_{\frac{1}{2}}(0)\leq Y_{\frac{1}{2}-\mu}^{-}(0)$. Letting $\xi_{0}\to\infty$, we arrive at a contradiction.

Now, suppose $Y(t_{n})-X_{\frac{1}{2}}(t_{n})\to-\infty$ as $n\to\infty$, which implies $Y^{+}_{\frac{1}{2}}(t_{n})-X_{\frac{1}{2}}(t_{n})\to-\infty$ as $n\to\infty$. Then, for any $\mu>0$ and $\xi_{0}\in\R$, we can find some $N=N(\mu,\xi_{0})>0$ such that $t_{N}<0$ and $v(t_{N},x)\leq u(t_{N},x-\xi_{0})+\mu$ for $x\in\R$. Applying Corollary \ref{cor-stability}, we find
\begin{equation*}
v(t,x)\leq u(t,x-\xi_{0}-\frac{A\mu}{\om})+\mu,\quad x\in\R,\quad t\geq t_{N}.
\end{equation*}
Setting $t=0$ in the above estimate, we find
\begin{equation*}
v(0,x)\leq u(0,x-\xi_{0}-\frac{A\mu}{\om})+\mu<u(0,X_{\frac{1}{2}}(0))+\mu=\frac{1}{2}+\mu,\quad\forall x>\xi_{0}+\frac{A\mu}{\om}+X_{\frac{1}{2}}(0),
\end{equation*}
which implies that $Y_{\frac{1}{2}+\mu}^{+}(0)\leq\xi_{0}+\frac{A\mu}{\om}+X_{\frac{1}{2}}(0)$. This leads to a contradiction if we let $\xi_{0}\to-\infty$. Hence, we have $\sup_{t\leq0}|Y(t)-X_{\frac{1}{2}}(t)|<\infty$. This completes the proof.
\end{proof}

Now, we prove Theorem \ref{thm-uniqeness}.

\begin{proof}[Proof of Theorem \ref{thm-uniqeness}]
Let $\theta_{3}\in(0,\min\{\theta_{0},1-\theta_{1}\})$. For $t_{0}\in\R$, we define
\begin{equation*}
\begin{split}
u^{-}(t_{0},x)&=u(t_{0},x-Y^{-}_{1-\theta_{3}}(t_{0})+X_{\theta_{3}}(t_{0}))-\theta_{3},\\
u^{+}(t_{0},x)&=u(t_{0},x-Y^{+}_{\theta_{3}}(t_{0})+X_{1-\theta_{3}}(t_{0}))+\theta_{3}.
\end{split}
\end{equation*}
We claim
\begin{equation*}
u^{-}(t_{0},x)\leq v(t_{0},x)\leq u^{+}(t_{0},x),\quad x\in\R.
\end{equation*}
In fact, if $x\geq Y^{-}_{1-\theta_{3}}(t_{0})$, then by monotonicity,
$$u^{-}(t_{0},x)\leq u(t_{0},X_{\theta_{3}}(t_{0}))-\theta_{3}=0<v(t_{0},x).$$
If $x<Y^{-}_{1-\theta_{3}}(t_{0})$, then by the definition of $Y^{-}_{1-\theta_{3}}(t_{0})$,
$$v(t_{0},x)>1-\theta_{3}>u^{-}(t_{0},x).$$
Hence, $u^{-}(t_{0},x)\leq v(t_{0},x)$. The inequality $v(t_{0},x)\leq u^{+}(t_{0},x)$ is checked similarly.

By Lemma \ref{lem-bounded-interface-width} and Lemma \ref{lem-bounded-interface-width-uniqueness-general}, we have
\begin{equation*}
L:=\max\bigg\{\sup_{t_{0}\in\R}|Y^{-}_{1-\theta_{3}}(t_{0})-X_{\theta_{3}}(t_{0})|,\sup_{t_{0}\in\R}|Y^{+}_{\theta_{3}}(t_{0})-X_{1-\theta_{3}}(t_{0})|\bigg\}<\infty.
\end{equation*}
Then, shifting $u^{-}(t_{0},x)$ to the left and $u^{+}(t_{0},x)$ to the right, we conclude from the monotonicity of $u(t,x)$ in $x$ that for all $t_{0}\in\R$, there holds
\begin{equation}\label{modified-intial-estimate}
u(t_{0},x+L)-\theta_{3}\leq u^{-}(t_{0},x)\leq v(t_{0},x)\leq u^{+}(t_{0},x)\leq u(t_{0},x-L)+\theta_{3}.
\end{equation}
That is, we are in the position to apply Theorem \ref{thm-stability}. So, we apply Theorem \ref{thm-stability} to \eqref{modified-intial-estimate} to conclude that there exist $t_{0}$-independent constants $C>0$ and $\om_{*}>0$, and a family of shifts $\{\xi_{t_{0}}\}_{t_{0}\in\R}\subset\R$ satisfying $\sup_{t_{0}\in\R}|\xi_{t_{0}}|<\infty$ such that
\begin{equation*}
\sup_{x\in\R}|v(t,x)-u(t,x-\xi_{t_{0}})|\leq Ce^{-\om_{*}(t-t_{0})}
\end{equation*}
for all $t\geq t_{0}$. We now pass to the limit $t_{0}\to-\infty$ along some subsequence to conclude $\xi_{t_{0}}\to\xi$ for some $\xi\in\R$, and then conclude that $v(t,x)=u(t,x-\xi)$ for all $(t,x)\in\R\times\R$. This completes the proof.
\end{proof}

\section{Periodicity and asymptotic speeds of  transition fronts}\label{sec-ptw}

In this section, we study the periodicity of transition fronts of  \eqref{main-eqn} under the additional time periodic assumption
on $f$, that is, there exists $T>0$ such that $f(t+T,u)=f(t,u)$ for all $t\in\R$ and $u\in[0,1]$. We also study asymptotic speeds of transition fronts of \eqref{main-eqn} under the additional {\it uniquely ergodic} assumption on $f$, that is,  the dynamical system $\{\si_{t}\}_{t\in\R}$ defined by
\begin{equation}\label{DS-shift-1}
\si_{t}: H(f)\ra H(f),\quad f\mapsto f(\cdot+t,\cdot)
\end{equation}
is compact (i.e., $H(f)$ is compact and metrizable) and uniquely ergodic, that is, $\{\si_{t}\}_{t\in\R}$ admits one and only one invariant measure, where
\begin{equation*}
H(f)=\overline{\{f(\cdot+t,\cdot):t\in\R\}}
\end{equation*}
with the closure taken under the open-compact topology (which is equivalent to locally uniform convergence in our case). Throughout this section, we assume (H1)-(H3).

Let $u(t,x)$ be a space non-increasing transition front of \eqref{main-eqn} with interface $X(t)$. The main results of this section are stated in the following theorem.

\begin{thm}\label{thm-asymptotic-speeds}
\begin{itemize}
\item[\rm(i)] Assume that $f(t,u)$ is $T$-periodic in $t$. Then,
$u(t,x)$ is a $T$-periodic traveling wave, that is, there are a constant $c>0$ and a function $\psi:\R\times\R\ra(0,1)$ satisfying
\begin{equation}\label{eqn-periodic-tw}
\begin{cases}
\psi_{t}=J\ast\psi-\psi+c\psi_{x}+f(t,\psi),\cr
\lim_{x\ra-\infty}\psi(t,x)=1,\,\,\lim_{x\ra\infty}\psi(t,x)=0\,\,\text{uniformly in}\,\,t\in\R,\cr
\psi(t,\cdot)=\psi(t+T,\cdot)\,\,\text{for all}\,\,t\in\R
\end{cases}
\end{equation}
such that $u(t,x)=\psi(t,x-ct)$ for all $(t,x)\in\R\times\R$.

\item[\rm(ii)] Assume that $f(t,u)$ is uniquely ergodic in $t$, and, in addition, twice continuously differentiable with
\begin{equation*}
\sup_{(t,u)\in\R\times[-1,2]}\big(|f_{tt}(t,u)|+|f_{tu}(t,u)|+|f_{uu}(t,u)|\big)<\infty.
\end{equation*}
Then, the asymptotic speeds $\lim_{t\to\pm\infty}\frac{X(t)}{t}$ exist.
\end{itemize}
\end{thm}

To prove Theorem \ref{thm-asymptotic-speeds}, let us first do some preparation. Note that if $f$ is periodic in $t$, then it is uniquely ergodic. In the rest of this section, we assume that $f(t,u)$ satisfies the assumptions in Theorem \ref{thm-asymptotic-speeds}(ii).

Observe that any $g\in H(f)$ satisfies (H2)-(H3) due to the regularity assumptions on $f(t,u)$. For any $g\in H(f)$, there is $t_n\to\infty$ such that $f(t+t_n,u)\to g(t,u)$ as $n\to\infty$ in open-compact topology.  By the regularity, without loss of generality, we may assume that there is $u^g(t,x)$ such that $u(t+t_n,x+X(t_n))\to  u^g(t,x)$ as $n\to\infty$ in open compact topology. It is not difficult to see that $u^g(t,x)$ is a space non-increasing transition front of
 \begin{equation}\label{main-eqn-g}
u_{t}=J\ast u-u+g(t,u).
\end{equation}
We may assume that $u^{g}(t,x)$ is the unique transition front of
\eqref{main-eqn-g} satisfying the normalization $X^{g}_{\frac{1}{2}}(0)=0$, where $X^{g}_{\frac{1}{2}}(t)$ is the interface location function of $u^{g}(t,x)$ at $\frac{1}{2}$, i.e., $u^{g}(t,X_{\frac{1}{2}}^{g}(t))=\frac{1}{2}$ for all $t\in\R$.

Let
\begin{equation}\label{profiles-asympto-spd-19318391}
\psi^{g}(t,x)=u^{g}(t,x+X^{g}_{\frac{1}{2}}(t)),\quad\forall (t,x)\in\R\times\R
\end{equation}
be the profile function of $u^{g}(t,x)$. Then, $\psi^{g}(t,0)=\frac{1}{2}$ for all $t\in\R$.

We prove

\begin{lem}\label{lem-asymptotic-speed}
There hold the following statements:
\begin{itemize}
\item[\rm(i)] for any $g\in H(f)$, there holds
$$\psi^{g}(t+\tau,x)=\psi^{g\cdot\tau}(t,x),\quad\forall(t,\tau,x)\in\R\times\R\times\R,$$
where $g\cdot\tau=g(\cdot+\tau,\cdot)$;

\item[\rm(ii)] there holds $\sup_{(t,\tau)\in\R\times\R}|\dot{X}_{\frac{1}{2}}^{f\cdot\tau}(t)|<\infty$;

\item[\rm(iii)] the limits
$$\lim_{x\to-\infty}\psi^{g}(t,x)=1\quad\text{and}\quad\lim_{x\to\infty}\psi^{g}(t,x)=0$$
are uniformly in $t\in\R$ and $g\in H(f)$;

\item[\rm(iv)] there holds $\sup_{g\in H(f)}\sup_{t\in\R}|\dot{X}^{g}_{\frac{1}{2}}(t)|<\infty$.
\end{itemize}
\end{lem}

We remark that $\rm(ii)$ is a special case of $\rm(iv)$, but it plays an important role in proving the lemma, so we state it explicitly.

\begin{proof}[Proof of Lemma \ref{lem-asymptotic-speed}]
For simplicity, we write $X^{g}(t)=X^{g}_{\frac{1}{2}}(t)$. Therefore, $u^{g}(t,X^{g}(t))=\frac{1}{2}$ and $X^{g}(0)=0$.

$\rm(i)$ Fix any $\tau\in\R$. We see that both
\begin{equation*}
u_{1}(t,x)=\psi^{g\cdot\tau}(t,x-X^{g\cdot\tau}(t))\quad\text{and}\quad u_{2}(t,x)=\psi^{g}(t+\tau,x-X^{g}(t+\tau))
\end{equation*}
are transition fronts of $u_{t}=J\ast u-u+g(t+\tau,x)$. Then, by uniqueness, i.e., Theorem \ref{thm-uniqeness}, there exists $\xi\in\R$ such that $u_{1}(t,x)=u_{2}(t,x+\xi)$. Moreover, since
\begin{equation*}
u_{1}(t,X^{g\cdot \tau}(t))=\psi^{g\cdot\tau}(t,0)=\frac{1}{2}\quad\text{and}\quad u_{2}(t,X^{g}(t+\tau))=\psi^{g}(t+\tau,0)=\frac{1}{2},
\end{equation*}
we find
$$u_{1}(t,X^{g}(t+\tau)-\xi)=u_{2}(t,X^{g}(t+\tau))=\frac{1}{2},$$
and hence, $X^{g\cdot \tau}(t)=X^{g}(t+\tau)-\xi$ by monotonicity. It then follows that
\begin{equation*}
\begin{split}
\psi^{g\cdot\tau}(t,x)&=u_{1}(t,x+X^{g\cdot\tau}(t))=u_{2}(t,x+X^{g\cdot\tau}(t)+\xi)\\
&=u_{2}(t,x+X^{g}(t+\tau))=\psi^{g}(t+\tau,x).
\end{split}
\end{equation*}

$\rm(ii)$ By $\rm(i)$, we in particular have
\begin{equation}\label{important-equality-f}
\psi^{f\cdot\tau}(t,x)=\psi^{f}(t+\tau,x),\quad\forall(t,\tau,x)\in\R\times\R\times\R.
\end{equation}
Since the limits $\psi^{f}(t,x)\to1$ as $x\to-\infty$ and $\psi^{f}(t,x)\to0$ as $x\to\infty$ are uniform in $t\in\R$, we find
\begin{equation}\label{uniform-limits-f-infinity}
\lim_{x\to-\infty}\psi^{f\cdot\tau}(t,x)=1\quad\text{and}\quad\lim_{x\to\infty}\psi^{f\cdot\tau}(t,x)=0\quad\text{uniformly in}\quad(t,\tau)\in\R\times\R.
\end{equation}

From \eqref{important-equality-f}, we also have
\begin{equation}\label{important-equality-f-u}
u^{f\cdot\tau}(t,x+X^{f\cdot\tau}(t))=u^{f}(t+\tau,x+X^{f}(t+\tau)),\quad\forall(t,\tau,x)\in\R\times\R\times\R.
\end{equation}
Setting $x=0$ and differentiating the resulting equality with respect to $t$, we find
\begin{equation*}
\begin{split}
\dot{X}^{f\cdot\tau}(t)&=\frac{\frac{d}{dt}[u^{f}(t+\tau,X^{f}(t+\tau))]-u^{f\cdot\tau}_{t}(t,X^{f\cdot\tau}(t))}{u^{f\cdot\tau}_{x}(t,X^{f\cdot\tau}(t))}\\
&=\frac{\frac{d}{dt}[u^{f}(t+\tau,X^{f}(t+\tau))]-u^{f\cdot\tau}_{t}(t,X^{f\cdot\tau}(t))}{u^{f}_{x}(t+\tau,X^{f}(t+\tau))},
\end{split}
\end{equation*}
where we used $u^{f\cdot\tau}_{x}(t,X^{f\cdot\tau}(t))=u^{f}_{x}(t+\tau,X^{f}(t+\tau))$, which comes from \eqref{important-equality-f-u}. We see that both $\frac{d}{dt}[u^{f}(t+\tau,X^{f}(t+\tau))]$ and $u^{f\cdot\tau}_{t}(t,X^{f\cdot\tau}(t))$ are bounded uniformly in $(t,\tau)\in\R\times\R$. Moreover, $u^{f}_{x}(t+\tau,X^{f}(t+\tau))$ is bounded uniformly in $(t,\tau)\in\R\times\R$ due to the uniform steepness, i.e., Lemma \ref{thm-uniform-steepness}. It then follows that $\sup_{(t,\tau)\in\R\times\R}|\dot{X}^{f\cdot\tau}(t)|<\infty$.

$\rm(iii)$ For any $g\in H(f)$, there is a sequence $\{t_{n}\}$ such that $g_{n}:=f\cdot t_{n}\to g$ in $H(f)$. Trivially, $\sup_{n}\sup_{(t,x)\in\R\times\R}|u^{g_{n}}_{t}(t,x)|<\infty$, and by $\rm(i)$, $\sup_{n}\sup_{(t,x)\in\R\times\R}|u^{g_{n}}_{x}(t,x)|<\infty$. It then follows from $\rm(ii)$ that
\begin{equation*}
\begin{split}
\sup_{n}\sup_{(t,x)\in\R\times\R}|\psi^{g_{n}}_{t}(t,x)|&=\sup_{n}\sup_{(t,x)\in\R\times\R}|u_{t}^{g_{n}}(t,x+X^{g_{n}}(t))+\dot{X}^{g_{n}}(t)u_{x}^{g_{n}}(t,x+X^{g_{n}}(t))|\\
&\leq\sup_{n}\sup_{(t,x)\in\R\times\R}|u_{t}^{g_{n}}(t,x+X^{g_{n}}(t))|\\
&\quad+\sup_{(t,\tau)\in\R\times\R}|\dot{X}^{f\cdot\tau}(t)|\sup_{n}\sup_{(t,x)\in\R\times\R}|u_{x}^{g_{n}}(t,x+X^{g_{n}}(t))|<\infty,\\
\sup_{n}\sup_{(t,x)\in\R\times\R}|\psi^{g_{n}}_{x}(t,x)|
&=\sup_{n}\sup_{(t,x)\in\R\times\R}|u_{x}^{g_{n}}(t,x+X^{g_{n}}(t))|<\infty.
\end{split}
\end{equation*}
In particular, by Arzel\`{a}-Ascoli theorem, there exists a continuous function $\psi(\cdot,\cdot;g):\R\times\R\to[0,1]$ such that $\lim_{n\to\infty}\psi^{g_{n}}(t,x)=\psi(t,x;g)$ locally uniformly in $(t,x)\in\R\times\R$. We then conclude from \eqref{uniform-limits-f-infinity} that
\begin{equation}\label{unfirm-limit-13579}
\lim_{x\to-\infty}\psi(t,x;g)=1\quad\text{and}\quad\lim_{x\to\infty}\psi(t,x;g)=0\quad\text{uniformly in}\,\,t\in\R\,\,\text{and}\,\,g\in H(f).
\end{equation}

It remains to show $\psi^{g}(t,x)=\psi(t,x;g)$. Fix any $g\in H(f)$. By $\rm(ii)$, there exists a continuous function $X(\cdot;g):\R\to\R$ such that, up to a subsequence,
\begin{equation}\label{convergence-13579}
\begin{split}
X^{g_{n}}(t)\to X(t;g)\quad\text{and}\quad\psi^{g_{n}}(t,x-X^{g_{n}}(t))\to\psi(t,x-X(t;g);g)
\end{split}
\end{equation}
as $n\to\infty$ locally uniformly in $(t,x)\in\R\times\R$. Since, trivially,
\begin{equation*}
\begin{split}
\sup_{n}\sup_{(t,x)\in\R\times\R}\bigg|\frac{d}{dt}\psi^{g_{n}}(t,x-X^{g_{n}}(t))\bigg|&=\sup_{n}\sup_{(t,x)\in\R\times\R}|u_{t}^{g_{n}}(t,x)|<\infty,\\
\sup_{n}\sup_{(t,x)\in\R\times\R}\bigg|\frac{d^{2}}{dt^{2}}\psi^{g_{n}}(t,x-X^{g_{n}}(t))\bigg|&=\sup_{n}\sup_{(t,x)\in\R\times\R}|u_{tt}^{g_{n}}(t,x)|<\infty,
\end{split}
\end{equation*}
we will also have
\begin{equation}\label{convergence-13579-time-derivative}
\frac{d}{dt}\psi^{g_{n}}(t,x-X^{g_{n}}(t))\to\frac{d}{dt}\psi(t,x-X(t;g);g)
\end{equation}
as $n\to\infty$ locally uniformly in $(t,x)\in\R\times\R$. Thus, $\psi(t,x-X(t;g);g)$ is a global-in-time solution of \eqref{main-eqn-g}, and hence, it is a transition front due to \eqref{unfirm-limit-13579}. Uniqueness of transition fronts and the normalization $X^{g_{n}}(0)=0$ then imply that $\psi^{g}(t,x)=\psi(t,x;g)$.

$\rm(iv)$ It is a simple consequence of $\rm(ii)$ and the proof of $\rm(iii)$.
\end{proof}

Now, we prove Theorem \ref{thm-asymptotic-speeds}.

\begin{proof}[Proof of Theorem \ref{thm-asymptotic-speeds}]
(i) By periodicity, $u(t+T,x)$ is also a transition front of \eqref{main-eqn}. Theorem \ref{thm-uniqeness} then yields the existence of some $\xi\in\R$ such that
\begin{equation}\label{pre-periodicity}
u(t+T,x)=u(t,x+\xi),\quad\forall(t,x)\in\R\times\R.
\end{equation}

Fix some $\theta_{*}\in(0,1)$. Setting $t=0$ and $x=X_{\theta_{*}}(T)$ in \eqref{pre-periodicity}, we find
$$\theta_{*}=u(T,X_{\theta_{*}}(T))=u(0,X_{\theta_{*}}(T)+\xi),$$
which leads to $X_{\theta_{*}}(0)=X_{\theta_{*}}(T)+\xi$ by monotonicity. It then follows from \eqref{pre-periodicity} that
$$u(t+T,x)=u(t,x+X_{\theta_{*}}(0)-X_{\theta_{*}}(T)),\quad\forall(t,x)\in\R\times\R.$$

Setting $c=\frac{X_{\theta_{*}}(T)-X_{\theta_{*}}(0)}{T}$ and $\psi(t,x)=u(t,x+ct)$ for $(t,x)\in\R\times\R$, we readily verify that $(c,\psi)$ satisfies \eqref{eqn-periodic-tw}. The fact that $c>0$ follows from the fact $u(t,x)$ moves to the right.

\medskip

(ii) Write $X^{g}(t)=X^{g}_{\frac{1}{2}}(t)$. Since $\sup_{t\in\R}|X^{f}(t)-X(t)|<\infty$, it suffices to show the existence of the limits $\lim_{t\to\pm\infty}\frac{X^{f}(t)}{t}$. Since
\begin{equation*}
\lim_{t\to\pm\infty}\frac{X^{f}(t)}{t}=\lim_{t\to\pm\infty}\frac{X^{f}(t)-X^{f}(0)}{t}=\lim_{t\to\pm\infty}\frac{1}{t}\int_{0}^{t}\dot{X}^{f}(s)ds,
\end{equation*}
we only need to show the dynamical system (i.e., the shift operators) generated by $\dot{X}^{f}(t)$ is compact and uniquely ergodic.

To this end, we first derive a formula for $\dot{X}^{g}(t)$. We claim
\begin{equation}\label{an-equality-13578}
\dot{X}^{g}(t)=-\frac{\int_{\R}J(y)\psi^{g}(t,-y)dy-\frac{1}{2}+g(t,\frac{1}{2})}{\psi^{g}_{x}(t,0)},\quad\forall t\in\R.
\end{equation}
In fact, differentiating $u^{g}(t,X^{g}(t))=\frac{1}{2}$, we find
\begin{equation*}
\dot{X}^{g}(t)=-\frac{u^{g}_{t}(t,X^{g}(t))}{u^{g}_{x}(t,X^{g}(t))}=-\frac{[J\ast u^{g}(t,\cdot)](X^{g}(t))-u^{g}(t,X^{g}(t))+f(t,u^{g}(t,X^{g}(t)))}{u^{g}_{x}(t,X^{g}(t))}.
\end{equation*}
The equality \eqref{an-equality-13578} then follows from $u^{g}(t,x+X^{g}(t))=\psi^{g}(t,x)$ and $u^{g}(t,X^{g}(t))=\frac{1}{2}$. Note that due to $\rm(i)$ in Lemma \ref{lem-asymptotic-speed} and \eqref{an-equality-13578}, there holds $\dot{X}^{g\cdot\tau}(t)=\dot{X}^{g}(t+\tau)$ for all $t,\tau\in\R$.

Next, we define
\begin{itemize}
\item the phase space $\tilde{H}=\{(\psi^{g},\dot{X}^{g})|g\in H(f)\}$;

\item the shift operators $\{\tilde{\si}\}_{t\in\R}$, i.e., the dynamical system on $\tilde{H}$,
\begin{equation*}
\tilde{\si}_{t}:\tilde{H}\ra\tilde{H},\quad (\psi^{g},\dot{X}^{g})\mapsto (\psi^{g\cdot t},\dot{X}^{g\cdot t})=(\psi^{g}(\cdot+t,\cdot),\dot{X}^{g}(\cdot+t));
\end{equation*}

\item an operator $\Om:H(f)\to\tilde{H}$, $g\mapsto(\psi^{g},\dot{X}^{g})$.
\end{itemize}
Clearly,
\begin{equation}\label{connecting-DSs}
\tilde{\si}_{t}\circ\Om=\Om\circ\si_{t},\quad\forall t\in\R,
\end{equation}
where $\{\si_{t}\}_{t\in\R}$ is given in \eqref{DS-shift-1}.

We show that $\Om$ is a homeomorphism. We first claim that $\Om$ is continuous. By \eqref{an-equality-13578}, the continuity of $\Om$ is the case if we can show that if $g_{n}\to g_{*}$ in $H(f)$ as $n\to\infty$, then
\begin{equation}\label{limit-to-show-13579}
\psi^{g_{n}}(t,x)\to\psi^{g_{*}}(t,x)\,\,\text{locally uniform in}\,\,t\in\R\,\,\text{and}\,\,\text{uniformly in}\,\,x\in\R
\end{equation}
as $n\to\infty$. To see this, let $g_{n}\to g_{*}$ in $H(f)$ as $n\to\infty$, then as in the proof of $\rm(iii)$ in Lemma \ref{lem-asymptotic-speed}, there exist continuous functions $X^{*}:\R\to\R$ and $\psi^{*}:\R\times\R\to[0,1]$ such that
\begin{equation*}
X^{g_{n}}(t)\to X^{*}(t)\,\,\text{and}\,\,\psi^{g_{n}}(t,x-X^{g_{n}}(t))\to\psi^{*}(t,x-X^{*}(t))\,\,\text{locally uniformly in}\,\,(t,x)\in\R\times\R
\end{equation*}
as $n\to\infty$. As in \eqref{convergence-13579-time-derivative}, we also have
\begin{equation*}
\frac{d}{dt}\psi^{g_{n}}(t,x-X^{g_{n}}(t))\to\frac{d}{dt}\psi^{*}(t,x-X^{*}(t))\,\,\text{locally uniformly in}\,\,(t,x)\in\R\times\R
\end{equation*}
as $n\to\infty$. In particular, $\psi^{*}(t,x-X^{*}(t))$ is global-in-time solution of \eqref{main-eqn-g} with $g$ replaced by $g^{*}$. Moreover, $\rm(iii)$ in Lemma \ref{lem-asymptotic-speed} forces $\psi^{*}(t,x-X^{*}(t))$ to be a transition front, and hence, $\psi^{*}(t,x)=\psi^{g_{*}}(t,x)$ by uniqueness and normalization. It then follows that $\psi^{g_{n}}(t,x)\to\psi^{g_{*}}(t,x)$ locally uniform in $(t,x)\in\R\times\R$ as $n\to\infty$. But, this actually leads to \eqref{limit-to-show-13579} due to the uniform limits as $x\to\pm\infty$ as in $\rm(iii)$ in Lemma \ref{lem-asymptotic-speed}. Hence, $\Om$ is continuous.

Clearly, from the continuity of $\Om$ and the compactness of $H(f)$, $\tilde{H}=\Om(H(f))$ is compact, and hence, $\tilde{H}=\overline{\{(\psi^{f\cdot t},\dot{X}^{f\cdot t})|t\in\R\}}$. Thus, if we can show that $\Om$ is one-to-one, then its inverse $\Om^{-1}$ exists and must be continuous, and hence, $\Om$ is a homeomorphism.

We show $\Om$ is one-to-one. For contradiction, suppose there are $g_{1},g_{2}\in H(f)$ with $g_{1}\neq g_{2}$, but $\Om g_{1}=\Om g_{2}$, i.e., $(\psi^{g_{1}},\dot{X}^{g_{1}})=(\psi^{g_{2}},\dot{X}^{g_{2}})$. In particular, $\dot{X}^{g_{1}}=\dot{X}^{g_{2}}$, which together with the normalization $X^{g_{1}}(0)=0=X^{g_{2}}(0)$ gives $X^{g_{1}}=X^{g_{2}}$. It then follows from \eqref{profiles-asympto-spd-19318391} that
\begin{equation*}
u^{g_{1}}(t,x)=u^{g_{2}}(t,x),\quad (t,x)\in\R\times\R,
\end{equation*}
which then leads to $g_{1}(t,u(t,x))=g_{2}(t,u(t,x))$, where $u=u^{g_{1}}=u^{g_{2}}$. Since $u(t,x)$ is continuous and connects $0$ and $1$ for any $t\in\R$, we conclude that $g_{1}=g_{2}$ on $\R\times[0,1]$. It is a contradiction. Hence, $\Om$ is one-to-one, and therefore, $\Om$ is a homeomorphism.

Since $\Om$ is a homeomorphism, invariant measures on $H(f)$ and $\tilde{H}$ are related by $\Om$. We then conclude from \eqref{connecting-DSs} and the fact $\{\si_{t}\}_{t\in\R}$ is compact and uniquely ergodic that $\{\tilde{\si}_{t}\}_{t\in\R}$ is compact and uniquely ergodic. Now, define $\Phi:\tilde{H}\to\R$ by setting $\Phi(\psi^{g},\dot{X}^{g})=\dot{X}^{g}(0)$. Clearly, $\Phi$ is continuous. We then conclude from the unique ergodicity of $\{\tilde{\si}_{t}\}_{t\in\R}$ that there exist constants $c^{\pm}=c^{\pm}(\Phi)\in\R$ such that
\begin{equation*}
\lim_{t\to\pm\infty}\frac{1}{t}\int_{0}^{t}\Phi(\tilde{\si}_{s}(\psi^{g},\dot{X}^{g}))ds=c^{\pm}
\end{equation*}
uniformly in $g\in H(f)$. In particular, $\lim_{t\to\pm\infty}\frac{1}{t}\int_{0}^{t}\Phi(\tilde{\si}_{s}(\psi^{f},\dot{X}^{f}))ds=c^{\pm}$. But,
\begin{equation*}
\Phi(\tilde{\si}_{s}(\psi^{f},\dot{X}^{f}))=\Phi((\psi^{f}(\cdot+s,\cdot),\dot{X}^{f}(\cdot+s)))=\dot{X}^{f}(s).
\end{equation*}
This completes the proof.
\end{proof}


\section{Existence  of space non-increasing transition fronts}\label{sec-construction-tf}

In this section, we investigate the existence of space non-increasing transition fronts of the equation \eqref{main-eqn}. Throughout this section, we assume (H1)-(H4).

The main result of this section is stated in the following theorem.

\begin{thm}\label{thm-transition-front}
Equation \eqref{main-eqn} admits a transition front $u(t,x)$ that is non-increasing in space.
\end{thm}
\begin{proof}
We use the perturbation method. Fix $0<\ep_{0}\ll1$. For $\ep\in(0,\ep_{0}]$, we consider the following perturbation of \eqref{main-eqn}
\begin{equation}\label{perturbed-main-eqn}
u_{t}=J\ast u-u+\ep u_{xx}+f(t,u),\quad (t,x)\in\R\times\R.
\end{equation}
The advantage of considering the above perturbed equation is that we are able to apply the methods in \cite{Sh99-2} (also see \cite{FCh,Ch97,Sh04}) to construct transition fronts of \eqref{perturbed-main-eqn}. Here, we are not going to repeat the construction since it is lengthy. We just point out that the construction highly relies on the instability of the stationary solution $u_{0}(t)$ of the ODE \eqref{ode-instable}. Thus, for each $\ep\in(0,\ep_{0}]$, equation \eqref{perturbed-main-eqn} admits space decreasing transition fronts. Moreover, direct adaption of the proof of Theorem \ref{thm-uniqeness} yields the uniqueness, up to space shifts, of transition fronts of \eqref{perturbed-main-eqn}. For each $\ep\in(0,\ep_{0}]$, let $u^{\ep}(t,x)$ be the unique transition front of \eqref{perturbed-main-eqn} satisfying the normalization $u^{\ep}(0,0)=\frac{1}{2}$. Also, from the construction, there also holds the uniform bounded interface width for $\{u^{\ep}(t,x)\}$, that is,
\begin{equation}\label{uniform-biw}
\forall\,\,0<\la_{1}\leq\la_{2}<1,\quad\sup_{\ep\in(0,\ep_{0}]}\sup_{t\in\R}\text{diam}\big\{x\in\R|\la_{1}\leq u^{\ep}(t,x)\leq\la_{2}\big\}<\infty.
\end{equation}
Note that if the sequence $\{u^{\ep}(t,x)\}$ converges to some solution of \eqref{main-eqn}, then this solution must be a transition front of \eqref{main-eqn} due to \eqref{uniform-biw}. However, the convergence of $\{u^{\ep}(t,x)\}$ to some solution of \eqref{main-eqn} is far from being clear, since we have no idea whether $u^{\ep}_{x}(t,x)$ and $u^{\ep}_{xx}(t,x)$ are locally bounded in $(t,x)$ and uniformly in $\ep$, that means, we can not simply pass to the limit $\ep\to0$ in \eqref{perturbed-main-eqn}. To circumvent this difficulty, we first consider solutions in some weak sense.

We see that for any $\phi\in C_0^\infty(\R)$, there holds
\begin{equation}\label{weak-solution-appro}
\begin{split}
\int_{\R} u^{\ep}(t,x)\phi(x)dx=&\int_{\R}u^{\ep}(0,x)\phi(x)dx+\int_0^t \int_{\R}[J*u^{\ep}-u^{\ep}]\phi(x)dxdt\\
&+\ep\int_0^t \int_{\R}u^{\ep}(t,x)\phi_{xx}(x)dxdt+\int_0^t\int_{\R}f(t,u^{\ep}(t,x))\phi(x)dxdt.
\end{split}
\end{equation}
To pass to the limit $\ep\to0$ in \eqref{weak-solution-appro}, we derive some convergence properties of $u^{\ep}(t,x)$.

Since $\{u^{\ep}(t,x)\}$ is pre-compact in $L^{1}_{loc}(\R\times\R)$ (see Lemma \ref{compactness-approximate} below), we can use the diagonal argument to find some $u\in L^{1}_{loc}(\R\times\R)$ and a sequence $\{\ep_{n}\}$ such that
\begin{equation}\label{appro-a.e.-converg}
\text{$u^{\ep_{n}}(t,x)\to u(t,x)$ for a.e. $(t,x)\in\R\times\R$ as $n\to\infty$.}
\end{equation}
Let $\Om\subset\R\times\R$ be a measurable set with Lebesgue measure zero such that
\begin{equation*}
\text{$u^{\ep_{n}}(t,x)\to u(t,x)$ pointwise in $(t,x)\in(\R\times\R)\backslash\Om$ as $n\to\infty$.}
\end{equation*}
Since the functions $\{u^{\ep_{n}}(0,x)\}$ are decreasing in $x$ and uniformly bounded, Helly's selection theorem implies that there exists a subsequence, still denoted by $\{\ep_{n}\}$, and a non-increasing function $v(0,\cdot)$ such that
\begin{equation}\label{convergence-at-0}
\text{$u^{\ep_{n}}(0,x)\to v(0,x)$ pointwise in $x\in\R$ as $n\to\infty$.}
\end{equation}
Fix $t\in\R\backslash\{0\}$. Again, by Helly's selection theorem, there exists a subsequence $\{\ep_{n_{k}}^{t}\}\subset\{\ep_{n}\}$ and a non-increasing function $v(t,\cdot)$ such that
\begin{equation*}
\text{$u^{\ep_{n_{k}}^{t}}(t,x)\to v(t,x)$ pointwise in $x\in\R$ as $k\to\infty$.}
\end{equation*}
Clearly, $u(t,x)=v(t,x)$ for $(t,x)\in(\R\times\R)\backslash\Om$. We then redefine $u(t,x)$ on $\Om$ to be $v(t,x)$. Hence, \eqref{appro-a.e.-converg} is still true, and, moreover, for any $t\in\R$, we have
\begin{equation}\label{appro-converg-on-section}
\text{$u^{\ep_{n_{k}}^{t}}(t,x)\to u(t,x)$ pointwise in $x\in\R$ as $k\to\infty$,}
\end{equation}
where $\{\ep_{n_{k}}^{0}\}=\{\ep_{n}\}$. Also, $u(t,x)$ is non-increasing in $x$.

Now, for fixed $t\in\R$, using \eqref{appro-a.e.-converg}, \eqref{convergence-at-0} and \eqref{appro-converg-on-section}, we pass to the limit $\ep\to0$ along the subsequence $\{\ep^{t}_{n_{k}}\}$ as $k\to\infty$ in \eqref{weak-solution-appro} to obtain
\begin{equation*}
\begin{split}
&\int_{\R} u(t,x)\phi(x)dx\\
&\quad\quad=\int_{\R}u(0,x)\phi(x)dx+\int_0^t \int_{\R}[J*u-u]\phi(x)dxdt+\int_0^t\int_{\R}f(t,u(t,x))\phi(x)dxdt\\
&\quad\quad=\int_{\R} \Big\{ u(0,x)+\int_0^t [J*u-u]d\tau+\int_0^t f(\tau,u(\tau,x)) dt\Big\}\phi(x)dx
\end{split}
\end{equation*}
for any $\phi\in C_0^\infty(\R)$. In particular, for any fixed $t\in\R$,
\begin{equation}\label{new-eq4}
u(t,x)=u(0,x)+\int_0^t [J*u-u+f(\tau,u(\tau,x))]d\tau \quad \text{for}\quad a.e.\,\, x\in\R.
\end{equation}
For $t\in\R$, let $\Om_{t}\subset\R$ be the measurable set with measure zero such that \eqref{new-eq4} is true for any $x\in\R\backslash\Om_{t}$. Note that $\R\backslash\Om_{t}$ is dense in $\R$, otherwise $\Om_{t}$ contains an open interval, which is impossible.

For $(t,x)\in\R\times\R$, define
$$
u^{*}(t,x)=\lim_{y\in\R\backslash\Om_{t},y\downarrow x}u(t,y).
$$
This is well-defined, since $u(t,x)$ is non-increasing in $x$ and $\R\backslash\Om_{t}$ is dense in $\R$. By \eqref{new-eq4}, we have that for any $t\in\R$,
$$
u^*(t,x)=u^*(0,x)+\int_0^t [J\ast u^*-u^{*}+f(\tau,u^*(\tau,x))]d\tau \quad \text{for all}\,\, x\in\R.
$$
This implies that $u^*(t,x)$ is continuous in $t$ uniformly with respect to $x\in\R$ and then
$$
u^{*}_{t}=J*u^{*}-u^{*}+f(t,u^{*})\quad \text{for all}\,\, (t,x)\in\R\times\R.
$$
We then conclude from \eqref{uniform-biw} that $u^{*}(t,x)$ is a transition front. Moreover, $u^{*}(t,x)$ is non-increasing in $x$, since $u(t,x)$ is non-increasing in $x$.
\end{proof}

To finish the proof of Theorem \ref{thm-transition-front} under assumption (H5), we prove the following lemma.

\begin{lem}\label{compactness-approximate}
$\{u^{\ep}(\cdot,\cdot)\}_{\ep\in(0,\ep_{0}]}$ is pre-compact in $L_{loc}^{1}(\R\times\R)$.
\end{lem}
\begin{proof}
We first show that
\begin{equation}\label{pre-compactness-1}
\text{$\{u^{\ep}(t,\cdot)\}_{\ep\in(0,\ep_{0}],t\in\R}$ is pre-compact in $L^{1}_{loc}(\R)$}.
\end{equation}
Since $u^{\ep}(t,\cdot)$ is a decreasing function for any $\ep\in(0,\ep_{0}]$ and $t\in\R$ and $\{u^{\ep}(t,\cdot)\}_{\ep\in(0,\ep_{0}],t\in\R}$ is uniformly bounded, Helly's selection theorem yields that for any sequence $(\ep_{n},t_{n})$ there exists a subsequence $(\ep_{n_{k}},t_{n_{k}})\subset(\ep_{n},t_{n})$ and a non-increasing function $v:\R\to[0,1]$ such that
\begin{equation*}
u^{\ep_{n_{k}}}(t_{n_{k}},x)\to v(x)\,\,\text{pointwise in}\,\,x\in\R\,\,\text{as}\,\,k\to\infty,
\end{equation*}
which, together with dominated convergence theorem and boundedness, imply that
\begin{equation*}
u^{\ep_{n_{k}}}(t_{n_{k}},\cdot)\to v\,\,\text{in}\,\,L^{1}_{loc}(\R)\,\,\text{as}\,\,k\to\infty.
\end{equation*}
This verifies \eqref{pre-compactness-1}.

Fix $r>0$ and let $B_{r}=(-r,r)\times(-r,r)$. It remains to show that $\{u^{\ep}(t,x)\}_{\ep\in(0,\ep_{0}]}$ restricted to $B_{r}$ is pre-compact in $L^{1}(B_{r})$. Due to the uniform boundedness of $\{u^{\ep}(t,x)\}_{\ep\in(0,\ep_{0}]}$, for any $\bar{\ep}>0$ we can find some $O\subset B_{r}$ such that
$$\int_{B_{r}\backslash O}u^{\ep}(t,x)dtdx\leq|B_{r}\backslash O|\leq\bar{\ep}.$$
Thus, applying \cite[Theorem 2.21]{Ad75}, it suffices to show that for any $\bar{\ep}>0$ there exists $\de>0$ such that
\begin{equation}\label{pre-compactness-2}
\int_{B_{r}}|u^{\ep}(t+\De t,x+\De x)-u^{\ep}(t,x)|dtdx\leq\bar{\ep}
\end{equation}
for all $\ep\in(0,\ep_{0}]$ and $(\De t,\De x)\in\R\times\R$ with $|\De t|+|\De x|\leq \de$.

To this end, let $\bar{\ep}>0$. For $(\De t,\De x)\in\R\times\R$, we have
\begin{equation}\label{some-estimate-152877529}
\begin{split}
&\int_{B_{r}}|u^{\ep}(t+\De t,x+\De x)-u^{\ep}(t,x)|dtdx\\
&\quad\quad\leq\int_{B_{r}}|u^{\ep}(t+\De t,x+\De x)-u^{\ep}(t+\De t,x)|dtdx+\int_{B_{r}}|u^{\ep}(t+\De t,x)-u^{\ep}(t,x)|dtdx.
\end{split}
\end{equation}
We use \eqref{pre-compactness-1} to control the first integral on the right hand side of \eqref{some-estimate-152877529}. In fact, by \eqref{pre-compactness-1} and \cite[Theorem 2.32]{Ad75}, for any $\bar{\ep}>0$ there exist $\de_{1}>0$ such that
\begin{equation*}
\int_{-r}^{r}|u^{\ep}(t,x+\De x)-u^{\ep}(t,x)|dx\leq\frac{\bar{\ep}}{4r}
\end{equation*}
for all $\ep\in(0,\ep_{0}]$, $t\in\R$ and $\De x\in\R$ with $|\De x|\leq \de_{1}$.
It then follows that
\begin{equation*}
\int_{B_{r}}|u^{\ep}(t+\De t,x+\De x)-u^{\ep}(t+\De t,x)|dtdx\leq\frac{\bar{\ep}}{2}
\end{equation*}
for all $\ep\in(0,\ep_{0}]$ and $(\De t,\De x)\in\R\times\R$ with $|\De x|\leq \de_{1}$.

For the second integral on the right hand side of \eqref{some-estimate-152877529}, we argue as in the proof of \cite[Lemma 2.5]{Sh03} to find some continuous and nondecreasing function $\al_{r}(\cdot)$ satisfying $\al_{r}(0)=0$ such that
\begin{equation*}
\int_{-r}^{r}|u^{\ep}(t+\De t,x)-u(t,x)|dx\leq \al_{r}(|\De t|)
\end{equation*}
for all $\ep\in(0,\ep_{0}]$ and $t\in\R$. It then follows that there exists $\de_{2}>0$ such that
\begin{equation*}
\int_{B_{r}}|u^{\ep}(t+\De t,x)-u^{\ep}(t,x)|dtdx\leq\frac{\bar{\ep}}{2}
\end{equation*}
for all $\ep\in(0,\ep_{0}]$ and $\De t\in\R$ with $|\De t|\leq \de_{2}$.

Now, setting $\de=\min\{\de_{1},\de_{2}\}$, we find \eqref{pre-compactness-2} and then completes the proof.
\end{proof}

\begin{rem}
We remark that transition fronts of \eqref{perturbed-main-eqn} can be constructed without the unbalanced condition \eqref{unbalanced-cond}. Hence, Theorem \ref{thm-transition-front} is true if we drop \eqref{unbalanced-cond}. But, in the absence of \eqref{unbalanced-cond}, the constructed transition front may not be continuous in space. We refer the reader to \cite{BaFiReWa97} for a sufficient and necessary condition for the existence of discontinuous traveling waves of $u_{t}=J\ast u-u+f_{B}(u)$. It would be interest and important to study the stability and uniqueness of transition fronts in the absence of \eqref{unbalanced-cond}.
\end{rem}

We end this paper by mentioning a variation on (H4). The point is that we allow the failure of \eqref{nondegeneracy-intro}. But, then, we need an additional assumption, that is, $u_{0}(t)\equiv\theta_{*}$ for some $\theta_{*}\in[\tilde{\theta},\theta]$. We assume

\begin{itemize}
\item[\bf (H5)]There exists $\theta_{*}\in[\tilde{\theta},\theta]$ such that
\begin{equation*}
f(t,u)<0,\,\, u\in(0,\theta_{*})\quad\text{and}\quad f(t,u)>0,\,\,u\in(\theta_{*},1)
\end{equation*}
for all $t\in\R$.
\end{itemize}

Using different techniques, we are able to prove Theorem \ref{thm-transition-front} under assumptions (H1)-(H3) and (H5). But, in this case, we can not drop the unbalanced condition \eqref{unbalanced-cond}.

\begin{proof}[Proof of Theorem \ref{thm-transition-front} under assumptions (H1)-(H3) and (H5)]
The proof can be done along the same line as that in the ignition case (see \cite{ShSh14-2}), so we here only outline the strategies within the following four steps.

\paragraph{\textbf{Step 1. Approximating front-like solutions.}} Let $\phi_{B}$ is the decreasing profile of bistable traveling waves of $u_{t}=J\ast u-u+f_{B}(u)$ with the normalization $\phi_{B}(0)=\theta_{*}$. For $s<0$ and $y\in\R$, denote by $u(t,x;s,\phi_{B}(\cdot-y))$ the classical solution of \eqref{main-eqn} with initial data $u(s,x;s,\phi_{B}(\cdot-y))=\phi_{B}(x-y)$. Then, it can be shown that for any $s<0$, there exists a unique $y_{s}\in\R$ such that $u(0,0;s,\phi_{B}(\cdot-y_{s}))=\theta$. Moreover, $y_{s}\to-\infty$ as $s\to-\infty$.

Set $u(t,x;s):=u(t,x;s,\phi_{B}(\cdot-y_{s}))$. We see that $u(t,x;s)$ is decreasing in $x$. The functions $\{u(t,x;s)\}_{s<0,t\geq s}$ are the approximating front-like solutions.

\paragraph{\textbf{Step 2. Bounded interface width-I.}} For $s<0$, $t\geq s$ and $\la\in(0,1)$, let $X_{\la}(t;s)$ be such that $u(t,X_{\la}(t;s);s)=\la$. It is well-defined and continuous in $t$.

Then, there exists $\la_{*}\in(\theta_{*},1)$ such that for any $\la_{1},\la_{2}\in(0,\la_{*}]$, there holds
\begin{equation*}
\sup_{s<0,t\geq s}\big|X_{\la_{1}}(t;s)-X_{\la_{2}}(t;s)\big|<\infty.
\end{equation*}
This is the difficult part in constructing transition fronts. Its proof is based on the rightward propagation estimate of $X_{\la}(t;s)$ and an idea of Zlato\v{s} (see \cite[Lemma 2.5]{Zl13}). It is important that $\la_{*}>\theta_{*}$, and it is the reason why we need $f(t,\theta_{*})=0$ for all $t\in\R$.

\paragraph{\textbf{Step 3. Bounded interface width-II.}} We extend the result in Step 2 to
\begin{equation*}
\forall\la_{1},\la_{2}\in(0,1),\quad \sup_{s<0,t\geq s}\big|X_{\la_{1}}(t;s)-X_{\la_{2}}(t;s)\big|<\infty.
\end{equation*}
It is done through the following two steps:
\begin{itemize}
\item[\rm(i)] there are $c_{\min}>0$, $c_{\max}>0$, $\tilde{c}_{\max}>0$ and $d_{\max}>0$ such that for any $s<0$, there exists a continuously differentiable function $X(t;s):[s,\infty)\ra\R$ satisfying
\begin{equation*}
c_{\min}\leq\dot{X}(t;s)\leq c_{\max}\quad\text{and}\quad|\ddot{X}(t;s)|\leq\tilde{c}_{\max}\quad\text{for}\quad t\geq s
\end{equation*}
such that $0\leq X(t;s)-X_{\la_{*}}(t;s)\leq d_{\max}$ for $t\geq s$; moreover, $\{\dot{X}(\cdot,s)\}_{s<0}$ and $\{\ddot{X}(\cdot,s)\}_{s<0}$ are uniformly bounded and uniformly Lipschitz continuous;

\item[\rm(ii)] using $\rm(i)$, we can find $c_{\pm}>0$ and $h_{\pm}>0$ such that
\begin{equation*}
u(t,x;s)\geq1-e^{c_{-}(x-X(t;s)+h_{-})}\quad\text{and}\quad
u(t,x;s)\leq e^{-c_{+}(x-X(t;s)-h_{+})}
\end{equation*}
for all $x\in\R$, $s<0$ and $t\geq s$.
\end{itemize}
Clearly, the bounded interface width follows.

\paragraph{\textbf{Step 4. Construction of transition fronts.}} The approximating solutions $\{u(t,x;s)\}_{s<0,t\geq s}$ converge to some transition front (as in Theorem \ref{thm-transition-front}) as $s\to-\infty$ along some subsequence due to the properties in Step 3 and the following: there holds
\begin{equation*}
\sup_{s<0,t\geq s}\sup_{x\neq y}\bigg|\frac{u(t,y;s)-u(t,x;s)}{y-x}\bigg|<\infty,
\end{equation*}
whose proof relies on the observation that for fixed $x$, the term $\frac{u(t,x+\eta;s)-u(t,x;s)}{\eta}$ for $0<|\eta|\leq\eta_{0}\ll1$ can only grow for a period of time that is independent of $x$.
\end{proof}

\appendix

\section{Comparison principles}\label{sec-app-cp}

We state comparison principles used in the previous sections. 

\begin{prop}\label{prop-app-comparison}
Let $K:\R\times\R\to[0,\infty)$ be continuous and satisfy $\sup_{x\in\R}\int_{\R}K(x,y)dy<\infty$.
Let  $a:\R\times\R\to\R$ be continuous and uniformly bounded.

\begin{itemize}
\item[\rm(i)] Suppose that $X:[0,\infty)\to\R$ is continuous and that $u:[0,\infty)\times\R\to\R$ satisfies the following:   $u, u_t:[0,\infty)\times\R\to\R$
are continuous, the limit $\lim_{x\ra\infty}u(t,x)=0$ is locally uniformly in $t$,  and
\begin{equation*}
\begin{cases}
u_{t}(t,x)\geq \int_{\R}K(x,y)u(t,y)dy+a(t,x)u(t,x),\quad x>X(t),\,\,t>0,\\
u(t,x)\geq0,\quad x\leq X(t),\,\,t>0,\\
u(0,x)=u_{0}(x)\geq0,\quad x\in\R.
\end{cases}
\end{equation*}
Then $u(t,x)\geq0$ for $(t,x)\in(0,\infty)\times\R$.

\item[\rm(ii)]  Suppose that $X:[0,\infty)\to\R$ is continuous and that $u:[0,\infty)\times\R\to\R$ satisfies the following:   $u, u_t:[0,\infty)\times\R\to\R$
are continuous, the limit $\lim_{x\ra -\infty}u(t,x)=0$ is locally uniformly in $t$,  and
\begin{equation*}
\begin{cases}
u_{t}(t,x)\geq \int_{\R}K(x,y)u(t,y)dy+a(t,x)u(t,x),\quad x<X(t),\,\,t>0,\\
u(t,x)\geq0,\quad x\ge X(t),\,\,t>0,\\
u(0,x)=u_{0}(x)\geq0,\quad x\in\R.
\end{cases}
\end{equation*}
Then $u(t,x)\geq0$ for $(t,x)\in(0,\infty)\times\R$.

\item[\rm(iii)] Suppose  that $u:[0,\infty)\times\R\to\R$ satisfies the following:   $u, u_t:[0,\infty)\times\R\to\R$
is continuous, $\inf_{t\ge 0,x\in\R}u(t,x)>-\infty$, and
\begin{equation*}
\begin{cases}
u_{t}(t,x)\geq \int_{\R}K(x,y)u(t,y)dy+a(t,x)u(t,x),\quad x\in\R,\,\,t>0,\\
u(0,x)=u_{0}(x)\geq0,\quad x\in\R.
\end{cases}
\end{equation*}
Then $u(t,x)\geq0$ for $(t,x)\in(0,\infty)\times\R$. Moreover, if $u_0(x)\not\equiv 0$, then
$u(t,x)>0$ for $(t,x)\in(0,\infty)\times\R$.
\end{itemize}
\end{prop}
\begin{proof}
See \cite[Proposition A.1]{ShSh14-2} for the proof.
\end{proof}

\begin{defn}
Let $t_{0}\in\R$ and $T>0$. A continuous function $u:[t_{0},t_{0}+T)\times\R\to\R$ is called a super-solution (or sub-solution) of \eqref{main-eqn} on $[t_{0},t_{0}+T)$ if $u(t,x)$ is differentiable in $t$ on $(t_{0},t_{0}+T)$ for any $x\in\R$ and satisfies
$$
u_{t}(t,x)\geq\text{(or $\leq$)}\int_{\R}J(x-y)u(t,y)dy+f(t,u(t,x)),\quad (t,x)\in(t_{0},t_{0}+T)\times\R.
$$
\end{defn}

Proposition \ref{prop-app-comparison}(iii) gives the following comparison principal for \eqref{main-eqn}.

\begin{prop}\label{prop-app-cp-nonlinear}
Let $t_{0}\in\R$ and $T>0$. Suppose $u^{+}(t,x)$ and $u^{-}(t,x)$ are super- and sub-solutions of \eqref{main-eqn} on $[t_{0},t_{0}+T)$, respectively. 
\begin{itemize}
\item[\rm(i)] If $u^{+}(t_{0},\cdot)\geq u^{-}(t_{0},\cdot)$, then $u^{+}(t,x)\geq u^{-}(t,x)$ for $(t,x)\in(t_{0},t_{0}+T)\times\R$.
\item[\rm(ii)] If $u^{+}(t_{0},\cdot)\geq u^{-}(t_{0},\cdot)$ and $u^{+}(t_{0},\cdot)\not\equiv u^{-}(t_{0},\cdot)$, then $u^{+}(t,x)>u^{-}(t,x)$ for $(t,x)\in(t_{0},t_{0}+T)\times\R$.
\end{itemize}
\end{prop}



\bibliographystyle{amsplain}

\end{document}